\documentclass{amsart}
\usepackage{graphicx}
\usepackage[usenames,dvipsnames]{xcolor}
\usepackage{amsthm}
\usepackage{amsfonts}
\usepackage{fancyhdr}
\usepackage{tikz}
\usepackage{hyperref}
\usepackage[normalem]{ulem}
\usepackage{caption}
\usepackage{cite}
\usepackage{subcaption}
\usepackage{amsmath}
\usepackage{amscd}
\usepackage{amssymb}
\usepackage{mathtools}
\usepackage{placeins}
\usepackage{verbatim}
\usepackage{color}
\usepackage{multirow}
\usepackage{enumitem}
\usepackage{booktabs}

\newcommand{\R}{{\mathbb{R}}}
\newcommand{\reals}{{\mathbb{R}}}
\newcommand{\Z}{{\mathbb{Z}}}
\newcommand{\integers}{{\mathbb{Z}}}
\newcommand{\Q}{{\mathbb{Q}}}
\newcommand{\rationals}{{\mathbb{Q}}}
\newcommand{\K}{{\mathbb{K}}}
\newcommand{\PP}{{\mathbb{P}}}
\newcommand{\Su}{{\mathbf{S}}}
\renewcommand{\S}{\Su}  
\newcommand{\M}{{\mathbf{M}}}

\DeclarePairedDelimiter\ceil{\lceil}{\rceil}
\DeclarePairedDelimiter\floor{\lfloor}{\rfloor}

\newcommand{\newword}[1]{\textbf{\textit{#1}}}
\newcommand{\g}{{\mathbf{g}}}

\renewcommand{\b}{{\mathbf{b}}}
\newcommand{\x}{{\mathbf{x}}}
\newcommand{\y}{{\mathbf{y}}}
\newcommand{\Sp}{{\mathbb{S}}}
\newcommand{\T}{{\mathbb{T}}}
\newcommand{\F}{{\mathcal{F}}}
\newcommand{\A}{{\mathcal{A}}}
\newcommand{\set}[1]{{\lbrace#1\rbrace}}
\newcommand{\arc}{\operatorname{arc}}

\newcommand{\curve}{\operatorname{curve}}
\newcommand{\univ}{{\operatorname{univ}}}
\newcommand{\notch}{^{\scalebox{0.6}{$\mathrel \blacktriangleright \joinrel \mathrel \blacktriangleleft$}}}
\newcommand{\br}[1]{{\langle #1 \rangle}}
\renewcommand{\r}{\mathtt{r}}
\renewcommand{\t}{\mathtt{t}}
\newcommand{\w}{{w}}
\newcommand{\cT}{\mathcal{T}}

\newcommand{\bullets}{\begin{itemize}}
\newcommand{\stellub}{\end{itemize}}

\newtheorem{theorem}{Theorem}[section]
\newtheorem{corollary}[theorem]{Corollary}
\newtheorem{lemma}[theorem]{Lemma}
\newtheorem{proposition}[theorem]{Proposition}
\newtheorem{conjecture}[theorem]{Conjecture}

\theoremstyle{definition}

\newtheorem{definition}[theorem]{Definition}
\newtheorem{remark}[theorem]{Remark}

\newtheorem{example}[theorem]{Example}

\usepackage{color}

\newcommand{\margincolor}{red}      
\definecolor{darkgreen}{rgb}{0,0.7,0}

\newcounter{margincounter}
\newcommand{\marginnum}{
\ifnum\value{margincounter}<10
\textcolor{\margincolor}{\begin{picture}(0,0)\put(2.2,2.4){\circle{9}}\end{picture}\footnotesize\arabic{margincounter}}
\else\ifnum\value{margincounter}<100
\textcolor{\margincolor}{\begin{picture}(0,0)\put(4.256,2.5){\circle{11}}\end{picture}\footnotesize\arabic{margincounter}}
\else
\textcolor{\margincolor}{\begin{picture}(0,0)\put(6.8,2.5){\circle{14}}\end{picture}\footnotesize\arabic{margincounter}}
\fi\fi
}

\thanks{Barnard, Meehan, and Viel were supported in part by NSF grants DMS-0943855 and CCF-1017217.
Reading was supported in part by NSF grant DMS-1101568.} 
\author[Barnard \and Meehan \and Reading \and Viel ]{Emily Barnard \and Emily Meehan \and Nathan Reading \and Shira Viel } 

\title[Coefficients for the four-punctured sphere]{Universal geometric coefficients for the four-punctured sphere}

\begin{document}
\maketitle
\begin{abstract}
We construct universal geometric coefficients for the cluster algebra associated to the four-punctured sphere and obtain, as a by-product, the $\g$-vectors of cluster variables.
We also construct the rational part of the mutation fan.
These constructions rely on a classification of the allowable curves (the curves which can appear in quasi-laminations).
The classification allows us to prove the Null Tangle Property for the four-punctured sphere, thus adding this surface to a short list of surfaces for which this property is known.
The Null Tangle Property then implies that the shear coordinates of allowable curves are the universal coefficients.
We compute shear coordinates explicitly to obtain universal geometric coefficients.
 \end{abstract}

\setcounter{tocdepth}{2}
\tableofcontents

\section{Introduction}\label{sec: introduction}
The fundamental combinatorial datum specifying a cluster algebra is an \newword{exchange matrix}: an $n\times n$ skew-symmetrizable (and in the special case considered here, skew-symmetric) integer matrix $B$.
The remaining data specifying a cluster algebra can be thought of as a choice of coefficients.
We consider the broad class of cluster algebras of geometric type, in which the choice of coefficients is a collection of vectors (\newword{geometric coefficients}), each with $n$ entries.
The usual definition of cluster algebras of geometric type is \cite[Definition~2.12]{ca4}, but we work with the broader definition given in \cite[Definition~2.8]{universal}, reviewed below in Section~\ref{cluster sec}.
The relationship between the two definitions is described in \cite[Remark~2.9]{universal}.

Universal geometric coefficients are a choice of coefficients that are universal in the sense that any other choice of geometric coefficients can be obtained from the universal geometric coefficients by coefficient specialization.
Constructing geometric coefficients is equivalent to constructing a basis, in the sense of \newword{mutation-linear algebra} with respect to $B$.
That is, we find a collection of vectors that is independent and spanning in a certain strong sense.
Closely related to this notion of basis is the \newword{mutation fan} $\F_B$, a complete fan that encodes the combinatorics and geometry of \newword{mutations} of $B$.
In particular, in many cases one can read off this basis from the rays of $\F_B$.

Certain cluster algebras are modeled by the combinatorics/geometry/topology of triangulated surfaces.
The surfaces model was advanced in \cite{fst, ft}, building on earlier work in \cite{gsv, fg1, fg2}.
More recently, explicit combinatorial formulas for the cluster variables in such algebras were given in \cite{msw}.
The goal of this paper is to construct universal geometric coefficients and describe the mutation fan for the exchange matrix $B$ encoded by a certain triangulation $T_0$ of a $2$-sphere. 
The triangulation $T_0$, a combinatorial tetrahedron, is shown in Figure~\ref{tri and mat}, 
along with its signed adjacency matrix $B=B(T_0)$.
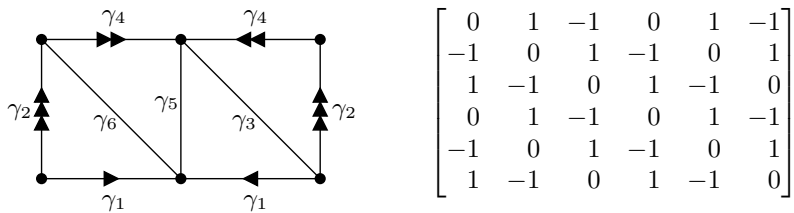
\begin{figure}\centering
\scalebox{1.05}{
\raisebox{-30pt}{\includegraphics{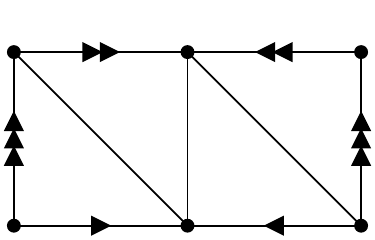}
\begin{picture}(0,0)(106,0)
\put(20,-6){\small$\gamma_1$}
\put(71,-6){\small$\gamma_1$}
\put(-14,27){\small$\gamma_2$}
\put(103,27){\small$\gamma_2$}
\put(17,23){\small$\gamma_6$}
\put(20,61){\small$\gamma_4$}
\put(71,61){\small$\gamma_4$}
\put(39,30){\small$\gamma_5$}
\put(67,23){\small$\gamma_3$}
\end{picture}}}
\qquad\quad
\raisebox{-1pt}{$
\begin{bmatrix*}[r]
0 & 1 & -1 & 0 & 1 & -1 \\
-1 & 0 & 1 & -1 & 0 & 1 \\
1 & -1 & 0 & 1 & -1 & 0 \\
0 & 1 & -1 & 0 & 1 & -1 \\
-1 & 0 & 1 & -1 & 0 & 1 \\
1 & -1 & 0 & 1 & -1 & 0
\end{bmatrix*}
$}
\caption{A triangulation $T_0$ of the four-punctured sphere and its signed adjacency matrix}
\label{tri and mat}
\end{figure}
(See Definition~\ref{def: signed adj}.)
Arrows in the picture indicate edge-identifications, and the arcs defining edges of $T_0$ are labeled $\gamma_i$.
We write $B$ with rows and columns indexed by $\gamma_1,\gamma_2,\gamma_3,\gamma_4,\gamma_5,\gamma_6$ in that order.
The four vertices of the triangulation are called punctures, because cluster algebras associated to $B$
can be understood in terms of the geometry of the four-punctured sphere, with $T_0$ taken to be an \emph{ideal} triangulation.
Universal coefficients for exchange matrices arising from triangulated surfaces are considered in general in \cite{unisurface}, 
and computed in the particular case of the once-punctured torus in \cite{unitorus}. 
The exchange matrix associated to the four-punctured sphere is the smallest example not yet addressed.

Geometric coefficients are encoded by the interactions between two closely related classes of curves: tagged arcs and allowable curves.
Both classes are equipped with a notion of compatibility.
Seeds in the cluster algebra correspond to tagged triangulations---maximal collections of pairwise compatible tagged arcs.
One example is the triangulation shown in Figure~\ref{tri and mat}.
The intersections between an allowable curve $\lambda$ and the arcs in a tagged triangulation $T$ are encoded in a vector $\b(T,\lambda)$ called the shear coordinates or shear coordinate vector of $\lambda$ with respect to $T$.
The rational quasi-lamination fan $\F_\rationals(T)$ is the fan whose cones are the nonnegative spans of the shear coordinate vectors of collections of pairwise compatible allowable curves with respect to $T$. 
Assuming a property of the surface called the Null Tangle Property, universal coefficients are the shear coordinates of allowable curves and the rational cones in the mutation fan $\F_{B(T)}$ coincide with the  
cones of $\F_\rationals(T)$. Our main results are the following theorems.
\begin{theorem}\label{sphere NTP}
The four-punctured sphere has the Null Tangle Property.
\end{theorem}

\begin{theorem}\label{unisphere thm}
Let $B$ be the exchange matrix shown in Figure~\ref{tri and mat}, arising from a triangulation of the four-punctured sphere.
Then the universal geometric coefficients for $B$ (over $\integers$ or $\rationals$) consist of all coordinate permutations by the group $\br{(14),\, (25),\, (36),\,  (123)(456)}$ of the vectors listed below, with $b/a$ varying over all standard forms of rational slopes with $0 < \frac{b}{a}\le\infty$.  
\[\begin{array}{ll}
1. &
\begin{bmatrix} -\floor*{\frac{b-1}{2}},& \floor*{\frac{a}{2}}+1,& \floor*{\frac{b-a}{2}},& -\floor*{\frac{b}{2}},& \floor*{\frac{a+1}{2}},& \floor*{\frac{b-a-1}{2}}  \end{bmatrix}
\\[4pt]
2. &
\begin{bmatrix} -\floor*{\frac{a}{2}}-1,& \floor*{\frac{b-1}{2}}, &\floor*{\frac{a-b+1}{2}},& -\floor*{\frac{a+1}{2}},& \floor*{\frac{b}{2}},& \floor*{\frac{a-b}{2}}+1  \end{bmatrix}
\\[4pt]
3. &
\begin{bmatrix}  -\floor*{\frac{b}{2}}, \floor*{\frac{a+1}{2}}, \floor*{\frac{b-a+1}{2}}, -\floor*{\frac{b+1}{2}}, \floor*{\frac{a}{2}}, \floor*{\frac{b-a}{2}}   \end{bmatrix}
\\[4pt]
4. &
\begin{bmatrix} -b,& a,& b-a,& -b,& a,& b-a \end{bmatrix}
\end{array}\]
\end{theorem}
Theorem~\ref{unisphere thm} has been stated for maximum simplicity, but does not list the universal geometric coefficients with maximum efficiency, because each vector listed has a stabilizer, under the permutation group $\br{(14),\, (25),\, (36),\,  (123)(456)}$, varying in size from $2$ to $8$.
Thus, while the group has $24$ elements, it produces only $6$ distinct permutations of the vectors arising from Item $1$ in the theorem and only $6$, $12$, and $3$ permutations, respectively, of the vectors arising from Items $2$--$4$. 
We list the vectors in a more complicated but more efficient way when we compute shear coordinates in Theorem~\ref{thm: shear}.
One can verify (for example, by modifying the easy argument in the last paragraph of the proof of \cite[Proposition~5.1]{unitorus}) that the vectors arising from Item 4 in the theorem are exactly the integer vectors in the plane given by the equations $x_1+x_2+x_3=0$, $x_1=x_4$, $x_2=x_5$, and $x_3=x_6$ whose entries have no common integer factors. 

To prove Theorems~\ref{sphere NTP} and~\ref{unisphere thm}, we classify tagged arcs and describe pairwise compatibility of tagged arcs.
The tagged arcs correspond bijectively to non-closed allowable curves, so we use the classification of tagged arcs to classify non-closed allowable curves and describe their compatibility. 
We then separately classify closed allowable curves and extend the description of compatibility.  
As an aside to the proofs of Theorems~\ref{sphere NTP} and~\ref{unisphere thm}, we also characterize tagged triangulations of the four-punctured sphere.

\begin{remark}\label{much appears}
Much of this curve/triangulation classification appears in the literature.
In \cite{FloydHatcher}, Floyd and Hatcher describe isotopy classes of arcs in the four-punctured sphere and describe maximal collections of non-intersecting curves. 
In \cite{tubular}, Barot and Geiss parametrize arcs in the four-punctured sphere by a rational slope and by the endpoints of the arcs.
Using this description of arcs, they then characterize tagged arc compatibility.
Also, \cite[Proposition~2.1]{KeenSeries} gives a parametrization of closed curves that coincides with ours, and describes that parametrization as well-known.  
(See also \cite{KomoriSeries}.)
Despite this literature, we give a different proof of the classification of arcs here, because it is essential for us to relate arcs in the four-punctured sphere with their lifts in the plane with respect to a certain covering map that will be described in Section~\ref{covering sec}.
The classification of closed curves from \cite{KeenSeries} already uses the covering map, but we give a proof here because the complete details are  not given in \cite{KeenSeries}, and because it is an easy consequence of the classification of arcs.
\end{remark}

Finally, we determine the shear coordinates of allowable curves.
Having explicit shear coordinates allows us to prove the Null Tangle Property (Theorem~\ref{sphere NTP}).
That property implies that our explicit shear coordinates are universal coefficients over~$\integers$ and $\rationals$, and we obtain Theorem~\ref{unisphere thm}.

Explicit shear coordinates and the description of compatibility among allowable curves lead to an explicit description of the rational part of the mutation fan for the four-punctured sphere.
(This is a fan in $\reals^6$.)
The \newword{$\g$-vector fan} for the cluster algebra associated to $B^T$ appears explicitly as a subfan of the mutation fan.  
Specifically, the $\g$-vectors of cluster variables for $B^T$ are written in terms of the shear coordinates of allowable non-closed curves and each cone of the $\g$-vector fan is given by the nonnegative linear span of a set of pairwise compatible allowable non-closed curves \cite[Proposition 5.2]{unisurface}.  
As a result, we obtain the following theorem.

\begin{theorem}\label{sphere g thm}
Let $B$ be the matrix of Theorem~\ref{unisphere thm}.
Then the $\g$-vectors of cluster variables associated to the transpose $B^T$ are the vectors described in Items $1$--$3$ of Theorem~\ref{unisphere thm}.
\end{theorem}

By determining the rational quasi-lamination fan, we know the rational cones of the mutation fan.
Conspicuously open is the question of what additional (irrational) cones exist in the mutation fan.
We conjecture that these irrational cones are all rays and that they are all located on one $2$-dimensional plane (Conjecture~\ref{irrat rays}).
That conjecture is the only missing piece in the proof of another conjecture (Conjecture~\ref{univ real}), which gives explicit universal coefficients over $\reals$ as well. 

We briefly review some of the background material related to geometric cluster algebras in Section~\ref{cluster sec}.
However, as already discussed, we work exclusively, in this paper, with surfaces, tagged arcs and allowable curves.
Thus we leave most of the details on cluster algebras of geometric type, universal geometric coefficients, mutation-linear algebra, mutation fans, etc.\ to \cite{universal} and most of the details of the connection between these concepts and surfaces to \cite{unisurface}.

The remainder of the paper is organized as follows.
In Section~\ref{defs sec}, we give the definitions and results required to construct universal coefficients and the mutation fan for the four-punctured sphere.
In Section~\ref{covering sec}, we describe, in general terms, how key topological facts about the four-punctured sphere, namely its relationship to the once-punctured and four-punctured tori and to the $\integers^2$-punctured plane, are used in this paper.
In Section~\ref{arc sec}, we classify and parametrize tagged arcs in the sphere and establish conditions for their pairwise compatibility in terms of this parametrization. 
In Section~\ref{tagged tri}, we catalog the tagged triangulations in the sphere, organized into six combinatorial types.
In Section~\ref{curve sec}, we use our work from the previous two sections to classify allowable curves in the sphere and describe maximal compatible collections thereof, corresponding to maximal ($6$-dimensional or $5$-dimensional) cones in the rational quasi-lamination fan $\F_\rationals(T_0)$.  
We also describe adjacencies among these maximal cones.
In Section~\ref{shear sec}, we explicitly compute shear coordinates of allowable curves with respect to $T_0$. 
In Section~\ref{NTP sec}, we prove that the four-punctured sphere has the Null Tangle Property.
Section~\ref{global sec} is devoted to a conjecture on the support of the rational quasi-lamination fan and the structure of the mutation fan, as well as a related conjecture on universal geometric coefficients over~$\reals$.


\section{Geometric cluster algebras}\label{cluster sec}
We provide a quick sketch of definitions from \cite[Section 2]{universal} for readers unfamiliar with cluster algebras of geometric type. 
Additional details can be found in \cite{ca4}, although the definition of geometric type here and in \cite{universal} is broader than in~\cite{ca4}.
(Taking $I$ to be finite and $R=\Z$ in what follows recovers \cite[Definition~2.2]{ca4}.) Throughout this section, $[n]$ stands for $\{1,2, \ldots, n\}$. 

Fix a positive integer $n$, an indexing set $I$, and choose an \newword{underlying ring} $R$ to be  $\Z, \Q$, or $\R$.
Define $\PP= \text{Trop}_R(u_i : i \in I)$ to be the \newword{tropical semifield} whose elements are formal products of the form $\prod_{i \in I} u_i^{a_i}$ for $a_i \in R$. Note that $\PP$ is determined up to isomorphism by $R$ and the cardinality of $I$.

A \newword{cluster algebra of geometric type} is a subring of the field $\K$ of rational functions in $n$ independent variables with coefficients in $\Q\PP$. 
The cluster algebra $\A_R(\x, \tilde{B})$ is uniquely defined by the choice of an initial
\newword{(labeled) geometric seed} $(\x, \tilde{B})$ consisting of an $n$-tuple $\x = (x_1, \ldots, x_n)$ of algebraically independent elements of $\K$ which generate the field and 
a function $\tilde{B}$ from $([n] \cup I) \times [n]$ to $R$ whose first $n$ rows form an exchange matrix $B$. 
Although $I$ may be infinite, for convenience we refer to $\tilde{B}$ as a matrix and call it an \newword{extended exchange matrix}. 
The rows indexed by $I$ are called \newword{coefficient rows} 
and define an $n$-tuple $\y=(y_1, \ldots, y_n)$ of elements in the coefficient semifield $\PP$ via $y_j = \prod_{i \in I}u_i^{b_{ij}}$.

The $n$-tuple $\x$ is called a \newword{cluster} and the rational functions $x_i$ are called \newword{cluster variables}. Each $k \in [n]$ defines an involution $\mu_k$ on the set of labeled geometric seeds of rank $n$, known as \newword{seed mutation}. 
Setting $\mu_k(\x,\tilde{B}) = (\x', \tilde{B}')$, the new cluster $(x_1', \ldots, x_n')$ has $x_j'=x_j$ for all $j \neq k$, 
defines $x_k'$ in terms of an \newword{exchange relation} involving $\x$ and $\y$, and defines $\tilde{B}'$ via \newword{matrix mutation} of $\tilde{B}$. 
In particular, each new cluster variable is a rational function (in fact, a Laurent polynomial) in the original cluster variables $x_i$ with coefficients in $\Q\PP$. 
The cluster algebra $\A_R(\x, \tilde{B})$ is then defined to be the subalgebra of $\K$ generated by all of the cluster variables in all of the seeds arising from performing all possible sequences of mutations on the initial seed $(\x, \tilde{B})$. 
Up to isomorphism, the cluster algebra is determined entirely by $\tilde{B}$, so we also refer to it as $\A_R(\tilde{B})$. 
 
Let $\A_R(\tilde{B})$ and $\A_R(\tilde{B}')$ be two cluster algebras with the same exchange matrix $B$ over the same underlying ring $R$. A ring homomorphism between them is called a \newword{coefficient specialization} if it maps the cluster variables and coefficients at each seed in $\A_R(\tilde{B})$ to the corresponding cluster variables and coefficients in $\A_R(\tilde{B}')$ while satisfying an additional technical requirement. 
(Namely, we require that the map be continuous with respect to a certain topology on $\PP$. 
This is, for example, the discrete topology if $I$ is finite, or the formal power series topology if $I$ is countably infinite.)
If for any geometric cluster algebra $\A_R(\tilde{B}')$ with exchange matrix $B$ there exists a unique coefficient specialization from $\A_R(\tilde{B})$ to $\A_R(\tilde{B}')$, then we say $\A_R(\tilde{B})$ is a 
\newword{universal geometric cluster algebra} over $B$ and call $\tilde{B}$ a \newword{universal extended exchange matrix}. The rows of $\tilde{B}$ are called 
\newword{universal geometric coefficients} for $B$ over $R$.

A ``more universal'' version of universal geometric coefficients was introduced in \cite[Section~12]{ca4} for cluster algebras of finite type.
Also introduced in \cite{ca4} is another special class of coefficients, the \newword{principal coefficients}.
Both universal and principal coefficients give rise to cluster algebras that can be considered ``largest,'' in some sense, among all (geometric) cluster algebras $\A$ with the same exchange matrix $B$. 
 In particular, \cite[Theorem~4.6]{ca4} states that the \newword{exchange graph} of any such $\A$ is covered by the exchange graph of the principal-coefficients cluster algebra $\A_\bullet(B)$.
Furthermore, \cite[Theorem~3.7]{ca4} describes an arbitrary cluster variable in $\A$ in terms of the corresponding cluster variable in $\A_\bullet(B)$. 
This description uses the ``auxiliary addition'' in the semifield $\PP$ associated to $\A$.
On the other hand, the cluster algebra $\A_\univ(B)$ with universal coefficients is ``largest'' in the sense that every $\A$ is a homomorphic image of $\A_\univ(B)$ by a homomorphism that respects the exchange graph.

 An important feature of universal geometric coefficients is the connection with mutation-linear algebra, and in particular with the mutation fan, which the authors believe to be a fundamental object.
This complete, usually infinite, fan consists of convex cones which are essentially the domains of linearity of the action of matrix mutation on coefficients.
As mentioned in Section~\ref{sec: introduction}, the mutation fan contains the $\g$-vector fan for $B^T$ as a subfan. 
In many cases, universal geometric coefficients may be obtained by taking one nonzero vector in each ray of the mutation fan.
These ideas are developed in detail in \cite{universal}.


\section{Surfaces, arcs, and curves}\label{defs sec}
In this section, we review the basic definitions from \cite{fst,ft,unisurface} of marked surfaces, arcs, triangulations, allowable curves, and so forth, as well as the results linking these constructions to the mutation fan and to universal coefficients. 
In this paper we consider three marked surfaces, the once-punctured torus, the four-punctured torus, and the four-punctured sphere, all marked surfaces without boundary. 
We quote the basic definitions and results as they apply to marked surfaces without boundary and then specialize to these surfaces.  
The reader interested in marked surfaces with boundary should consult the original definitions. 

In general, a \newword{marked surface} without boundary is a pair $(\Su, \M)$ where $\Su$ is a compact orientable surface without boundary and $\M$ is a set of points in $\Su$ called \newword{marked points}.
The pair $(\Su, \M)$ must satisfy certain conditions which we do not review here (see \cite[Definition~2.1]{fst}, \cite[Definition~3.1]{unisurface}). The constructions we consider are essentially concerned with the topology of the complement $\Su\setminus\M$, so, as usual, we call the marked points \newword{punctures} and refer to the sphere with 4 marked points as the \newword{four-punctured sphere}, which we denote by $\Sp-4$. 

In \cite{fst,ft}, there are two notions of arcs, namely tagged arcs and (ordinary) arcs, with slightly different topological constraints in their definitions.
For us, the ordinary arcs will not be needed.  
Instead, we define \emph{taggable} arcs and tagged arcs. 
The taggable arcs are those arcs which can be made into tagged arcs by adding taggings.

\begin{definition}
\label{def: arcs}  
A \newword{taggable arc} in $(\Su, \M)$ is a curve $\gamma$ in $\Su$ such that:
\bullets
\item The endpoints of $\gamma$ lie in the set $\M$.
\item $\gamma$ is non-self-intersecting, except possibly at coinciding endpoints. In the case when its endpoints coincide, we call $\gamma$ a \newword{loop}. \item $\gamma$ has no punctures in its interior.  (That is, $\gamma$ intersects $\M$ only at its endpoints.)  
\item $\gamma$ does not bound an unpunctured monogon or a once-punctured monogon.
\stellub
We consider taggable arcs up to isotopy relative to $\M$.
Two taggable arcs are \newword{compatible} if there is an isotopy representative of each such that the two representatives do not intersect, except possibly at their endpoints.
\end{definition}

\begin{lemma}\label{no loops}
There are no taggable loops in the four-punctured sphere.
\end{lemma}
\begin{proof}
A loop separates the sphere into two disks, one of which contains $0$ or $1$ punctures, but both possibilities are specifically excluded for taggable arcs in Definition~\ref{def: arcs}.
\end{proof}

\begin{definition}
\label{def: tag}
A \newword{tagged arc} $\gamma$ in $(\Su, \M)$ is a taggable arc in $(\Su,\M)$ 
together with a designation (or ``tagging'') of each end of the arc as either \newword{plain} or \newword{notched}.  
If the arc is a loop, then it must have the same tagging at both ends.
Two tagged arcs are \newword{compatible} if one of the following conditions holds:   
\bullets
\item The two underlying taggable arcs are the same (up to isotopy), and their tagging agrees at one endpoint but not the other. 
\item The two underlying taggable arcs are distinct and compatible (up to isotopy), and any shared endpoints have the same tagging.
\stellub
A \newword{tagged triangulation} of $(\Su,\M)$ is a maximal collection of distinct pairwise compatible tagged arcs. 
\end{definition}


\begin{definition}\label{def: signed adj}
Suppose $T$ is a tagged triangulation of $(\S,\M)$.
The \newword{signed adjacency matrix} $B(T)$ encodes the structure of $T$.
The rows and columns of $B(T)$ are indexed by the arcs in $T$.
Here, it is only necessary to define $B(T)$ in the case where all of the arcs of $T$ are tagged plain everywhere (for a general treatment, see \cite[Definition~4.1]{fst}). In this case, $T$ defines a decomposition of $\Su$ into triangles, with each arc shared by two different triangles.
For each pair $\alpha,\beta$ of arcs in $T$, the $\alpha\beta$-entry of $B(T)$ is the sum of contributions from each triangle of $T$.
A triangle contributes $+1$ if $\alpha$ and $\beta$ are arcs of the triangle with $\alpha$ immediately preceding $\beta$ in a clockwise traversal of the triangle, or contributes $-1$ if the same is true for a counterclockwise traversal.
The contribution is $0$ otherwise.
An example of a tagged triangulation, with all arcs tagged plain, and its corresponding signed adjacency matrix appear in Figure~\ref{tri and mat}.
(We will see in Proposition~\ref{all tri} that there is only one other combinatorial type of triangulation of the four-punctured sphere with all tags plain.)
\end{definition}

The signed adjacency matrix $B(T)$ is skew-symmetric, and thus each tagged triangulation encodes the exchange matrix of a cluster algebra of geometric type.
Passing from one tagged triangulation to another by ``flips'' of tagged arcs corresponds to matrix mutation.
(This is \cite[Lemma~9.7]{fst}.
See the discussion following Proposition~\ref{prop: maximal curves} for more on flipping tagged arcs.)
Another class of curves called allowable curves encodes the choice of coefficient rows.
We review their definition below.

\begin{definition}\label{def: allow}
An \newword{allowable curve} in $(\Su, \M)$ is a non-self-intersecting curve in $\Su \setminus \M$ that either  
\bullets
\item spirals into a puncture at both ends, or 
\item is closed, 
\stellub  
and 
\bullets
\item is not contractible in $\Su\setminus\M$, 
\item is not contractible to a point in $\M$, and
\item if it spirals into the same point at both ends, then it does not cut out a once-punctured disk.
\stellub

We consider allowable curves up to isotopy.
A curve may spiral into a puncture in a clockwise or counterclockwise manner, and the spiral direction may be chosen independently at each spiral point.
If the curve spirals into the same puncture at both ends, then the requirement that the curve not intersect itself forces both spirals to be in the same direction.  
Curves that coincide except for spiral directions are considered to be different curves.
Two allowable curves are \newword{compatible} either if there is an isotopy representative of each such that the two representatives do not intersect, 
or if the two curves are the same (up to isotopy) except that their spiral direction disagrees at exactly one end.
An \newword{(integral) quasi-lamination} is a collection of pairwise compatible allowable curves with positive integer weights. 
\end{definition}

Given a tagged arc $\alpha$ in $(\Su, \M)$ there is a unique allowable curve $\kappa(\alpha)$ which coincides with $\alpha$ except that, at each endpoint $v$, if $\alpha$ is tagged plain at $v$, then $\kappa(\alpha)$ spirals clockwise into $v$, and if $\alpha$ is tagged notched at $v$, then $\kappa(\alpha)$ spirals counterclockwise into $v$. 
It is immediate from the definitions that $\kappa$ is a bijection from tagged arcs to allowable curves that are not closed, and that two tagged arcs $\alpha$ and $\gamma$ are compatible if and only if the allowable curves $\kappa(\alpha)$ and $\kappa(\gamma)$ are compatible. 
(See also \cite[Definition~17.2]{fst} and \cite[Lemma~5.1]{unisurface}.)
Thus $\kappa$ takes tagged triangulations to maximal collections of non-closed pairwise compatible allowable curves.
Despite this partial correspondence via the map $\kappa$, it is useful to consider allowable curves and quasi-laminations separately from tagged arcs and tagged triangulations in order to define shear coordinates. 

\begin{definition}\label{shear def}
Given a quasi-lamination $L$ and a tagged triangulation $T$,  the \newword{shear coordinate vector} (or \newword{shear coordinates}) of $L$ with respect to $T$ is a vector $\b(T,L)=(b_\gamma(T,L):\gamma\in T)$ indexed by the tagged arcs $\gamma$ in $T$.
We define shear coordinates here only in the case when all arcs incident to a single vertex of $T$ are tagged identically and begin with the case when all of the tagged arcs are tagged plain everywhere.
The shear coordinates record how a 
quasi-lamination interacts with each arc and the two triangles that share that arc.
Each allowable curve $\lambda\in L$ contributes a quantity $\b(T,\lambda)$, and $\b(T,L)$ is the sum $\sum_{\lambda\in L}w_\lambda \b(T,\lambda)$, where $w_\lambda$ is the weight of $\lambda$ in $L$.
To compute $b_\gamma(T,\lambda)$, we apply an isotopy to $\lambda$ so as to minimize the number of intersection points of $\lambda$ with~$\gamma$.
Then $b_\gamma(T,\lambda)$ is the sum, over each intersection of $\lambda$ with $\gamma$, of a number that is~$-1$, $0$, or $1$.
The number is $\pm1$ as shown in Figure~\ref{shear fig} if $\lambda$ intersects the two triangles as shown, and $0$ otherwise.
The arc $\gamma$ is the diagonal of the square pictured and part of the curve $\lambda$ is represented by a vertical or horizontal dashed line.

If $T$ has some arcs notched (still assuming that for each vertex, all taggings at that vertex agree), then $\b(T,L)$ is defined to equal $\b(T',L')$, where $T'$ and $L'$ are a new tagged triangulation and a new 
quasi-lamination that we now define.
The triangulation $T'$ agrees with $T$ except that every notched tagging has been changed to plain.
The quasi-lamination $L'$ agrees with $L$ except that, for every vertex $v$ that had notched taggings in $T$, the spiral directions of all curves in $L$ incident to $v$ are reversed at $v$.
We then calculate $\b(T',L')$ as described in the previous paragraph.  
\end{definition}
\begin{figure}[ht]\centering
\scalebox{0.7}{\raisebox{41 pt}{\large$+1$}\,\,\,\,\includegraphics{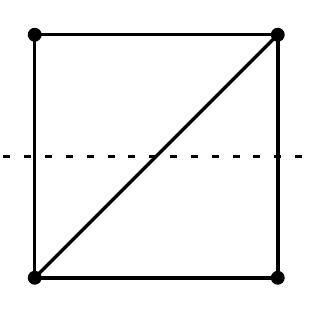}}
\qquad\qquad
\scalebox{0.7}{\includegraphics{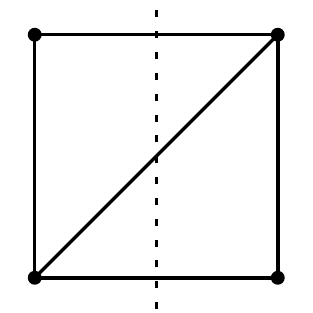}\,\,\raisebox{41 pt}{\large$-1$}}
\caption{Shear coordinates}
\label{shear fig}
\end{figure}

The following is \cite[Theorem~4.4]{unisurface}, which is a straightforward modification of \cite[Theorem~13.6]{ft}.

\begin{theorem}\label{q-lam bij}
Fix a tagged triangulation $T$.
Then the map $L\mapsto\b(T,L)$ is a bijection between integral (resp. rational) quasi-laminations and $\integers^n$ (resp.~$\rationals^n$).
\end{theorem}

\begin{definition}\label{rat q-lam fan def}
Fix a tagged triangulation $T$.
For any collection $\Lambda$ of pairwise compatible allowable curves, 
let $C_\Lambda$ be the nonnegative $\reals$-linear span of the shear coordinates $\set{\b(T,\lambda):\lambda\in\Lambda}$.
Let $\F_\rationals(T)$ be the collection of all cones $C_\Lambda$.
It is shown in \cite[Theorem~4.10]{unisurface} that $\F_\rationals(T)$ is a rational simplicial fan.
We call $\F_\rationals(T)$ the \newword{rational quasi-lamination fan} for $T$.
\end{definition}

The key property of a marked surface that enables results about the mutation fan and  universal coefficients is the Null Tangle Property, which we now define.
\begin{definition}
\label{def tangle}
A \newword{tangle} in $(\Su,\M)$ is an arbitrary finite collection $\Xi$ of allowable curves, each with an integer weight.
In contrast with the definition of a quasi-lamination, there is no requirement of pairwise compatibility of curves and no requirement of positivity of weights.
The shear coordinates $\b(T,\Xi)$ of a tangle $\Xi$ with respect to a tagged triangulation $T$ are computed as the weighted sum of the shear coordinates, with respect to $T$, of the curves in the tangle.
The \newword{support} of a tangle is the set of curves appearing in the tangle with non-zero weight.
A tangle is \newword{trivial} if its support is empty.
A \newword{null tangle} is a tangle $\Xi$ such that, for any tagged triangulation $T$, the shear coordinates $\b(T,\Xi)$ are zero.
The \newword{Null Tangle Property} is the property of a marked surface that if $\Xi$ is a null tangle then it is trivial.
\end{definition}
The Null Tangle Property was proved in \cite[Theorem~7.4]{unisurface} for a family of surfaces coincidentally equal to the surfaces of finite growth (see \cite[Section~11]{fst}), and in \cite{unitorus} for the once-punctured torus.
We prove it in this work for the four-punctured sphere (Theorem~\ref{sphere NTP}).

The following theorem is the combination of parts of \cite[Theorem~4.10]{unisurface}, \cite[Proposition~6.10]{universal}, \cite[Theorem~8.7]{universal}, \cite[Proposition~7.10]{fst}, \cite[Theorem~7.3]{unisurface}, and \cite[Theorem~4.4]{universal}, along with a simple argument that we give below.
For the relevant definitions, see \cite[Section~ 2]{unisurface}.

\begin{theorem}\label{Null Tangle consequences}
Suppose $(\S,\M)$ has the Null Tangle Property and let $T$ be any tagged triangulation of $(\S,\M)$.
Then  
\bullets
\item The fan $\F_\rationals(T)$ is the rational part of the mutation fan $\F_{B(T)}$.  In particular:  
\bullets
\item  The support of $\F_\rationals(T)$ contains $\rationals^n$.  
\item  Every rational cone in $\F_{B(T)}$ 
is a cone in $\F_\rationals(T)$.
\item  The full-dimensional cones in $\F_{B(T)}$ are exactly the full-dimensional cones in $\F_\rationals(T)$.
\item  Every codimension-1 cone in $\F_\rationals(T)$ is a cone in $\F_{B(T)}$. 
\stellub
\item The $\g$-vector fan of the transpose $B(T)^T$ is the set of all cones $C_\Lambda$ such that $\Lambda$ is a collection of pairwise compatible \emph{non-closed} allowable curves.  (This statement must be modified if $(\S,\M)$ is a marked surface without boundary having exactly one puncture, but we will not need that case here.)
\item The shear coordinates of allowable curves constitute universal geometric coefficients for $B(T)$. 
\stellub
\end{theorem}
Every part of this theorem follows immediately from the results listed above, except for the assertion that every codimension-1 cone in $\F_\rationals(T)$ is a cone in $\F_{B(T)}$.
The latter assertion follows from the others by an easy argument that we give here.
Saying that $\F_\rationals(T)$ is the rational part of $\F_{B(T)}$ is saying that (i) every cone in $\F_\rationals(T)$ is a rational polyhedral cone, (ii) each cone $C$ of $\F_\rationals(T)$ is contained in a cone $C'$ in $\F_{B(T)}$, and (iii) for every cone $C'$ of $\F_{B(T)}$, there is a unique largest (in the sense of containment) cone $C$ among cones in $\F_\rationals(T)$ contained in $C'$, and $C$ contains all rational points of $C'$.
If $C$ is a codimension-1 cone in $\F_\rationals(T)$, let $C'$ be the smallest cone in $\F_{B(T)}$ containing $C$.
If $C'$ is of codimension-1, then $C$ is the largest cone of $\F_\rationals(T)$ contained in $C'$, so $C$ contains every rational point in $C'$.
Since $C'$ and $C$ are of the same dimension, we conclude that they are equal.
If $C'$ is full-dimensional, then it is a cone in $\F_\rationals(T)$.  
But then $C$ is a face of $C'$, so we chose $C'$ incorrectly.
We conclude that $C$ is a face of $\F_{B(T)}$.
(The analogous argument holds in the context of \cite[Proposition~6.10]{universal}, so this is a general property of the ``$R$-part'' of a fan as defined in \cite[Section~6]{universal}.)

\section{Sphere, torus, and plane}\label{covering sec} 

Among the main tools in our arguments are some well-known covering spaces of the four-punctured sphere $\Sp-4$.
We will use these covering spaces to borrow results from the classification of arcs given in \cite{unitorus}.

First, $\Sp-4$ is doubly-covered by the four-punctured torus $\T-4$.
One way to picture the covering is to skewer a donut so that the surface of the donut is punctured in four places and so that rotating the skewer through a half-turn is a symmetry of the donut.
The quotient of the four-punctured torus modulo this rotation is a four-punctured sphere.

Second, the $\integers^2$-punctured plane $\reals^2-\integers^2$ is a covering space of $\Sp-4$. 
The covering of $\Sp-4$ by $\reals^2-\integers^2$ has the following description, which also relates to the covering of $\Sp-4$ by $\T-4$.
Consider the group $G$ of rigid motions on $\reals^2$ generated by half-turns about integer points.
This group consists of all such rotations together with all even-integer translations (i.e.\ translations taking $(0,0)$ to points with even-integer coordinates).
The even-integer translations constitute a normal subgroup $H$ of $G$ of index $2$.  
The quotient $(\reals^2-\integers^2)/H$ is $\T-4$ and the quotient $(\reals^2-\integers^2)/G$ is $\Sp-4$.
In each case, the natural map to the quotient is a covering map.  
The quotient group $G/H$  has two elements.
Its nontrivial element $\sigma$ acts on $\T-4$ by the rotation of the ``skewer,'' as discussed above, so that  $(\T-4)/(G/H)$ is $\Sp-4$. 
Moreover, the covering map from $(\reals^2-\integers^2)/G$ to $\Sp-4$ factors as the composition of the covering map from $\reals^2-\integers^2$ to $\T-4$ followed by the covering map from $\T-4$ to $\Sp-4$. 

\begin{figure}[ht]\centering
\scalebox{1}{\includegraphics{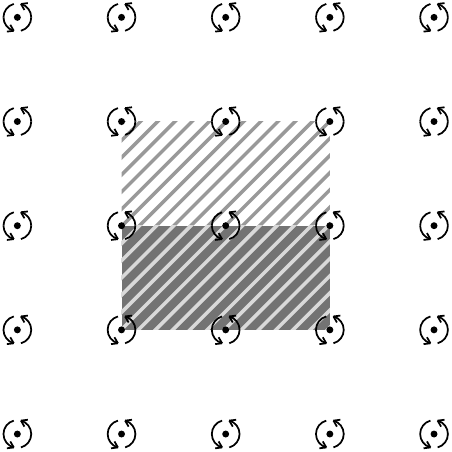}}
\caption{Fundamental domains for $G$ and $H$ on $\reals^2-\integers^2$}
\label{domain fig}
\end{figure}

Each covering map extends continuously to a map between unpunctured surfaces, but the extended maps to $\Sp-4$ are branched coverings, not coverings.
Although we will avoid explicitly dealing with branched covers, we will implicitly use the extended maps to make reference to images and preimages of punctures under the covering maps.
Using the (unbranched) covering maps is consistent with thinking of the punctures as ``deleted points,'' whereas using the extended maps is consistent with thinking of the punctures as ``marked points'' as they are formally called. 
(See Section~\ref{defs sec}.)
As there should be no confusion, we pass freely between these two points of view.

These constructions are pictured in Figures~\ref{domain fig}, \ref{modH fig}, and~\ref{modG fig}.
Figure~\ref{domain fig} shows $\reals^2$ with the points $\integers^2$ marked by dots and an indication of half-turns.
A fundamental domain for the action of $G$ is indicated by a dark gray rectangle, and a fundamental domain for $H$ is indicated by the entire striped square (including the dark gray rectangle).
Figure~\ref{modH fig} depicts $(\reals^2-\integers^2)/H$ as the striped square with edge identifications marked with arrows.
The points marked by dots are deleted and, after the edge identifications, these points become the four punctures of the torus labeled as $v_{00}, v_{01}, v_{10},$ and $v_{11}$. 
(For $ij\in\set{00,01,10,11}$, the preimage of $v_{ij}$ under the covering map is the set of points $[a,b]\in\integers^2$ with $[a,b]\equiv[i,j]$ modulo~ 2.)
The unlabeled center point in Figure~\ref{modH fig} is $v_{11}$.
As indicated in the picture, the unique non-identity element of $G/H$ acts on $(\reals^2-\integers^2)/H$ by a half-turn of the square. 
This action coincides with a half-turn about any of the punctures, as is apparent in the skewered-torus picture or by observing that $(\reals^2-\integers^2)/H$ can be pictured as a square, with identifications, centered at any of the punctures.
Figure~\ref{modG fig} depicts~$(\reals^2-\integers^2)/G$ as the gray rectangle with edge identifications marked with arrows and deleted points, which become the punctures, again labeled $v_{00}$, $v_{01}$, $v_{10}$, and~$v_{11}$.

\begin{figure}[ht]\centering
\begin{minipage}{240pt}\centering
\vspace{5pt}
\scalebox{1}{\includegraphics{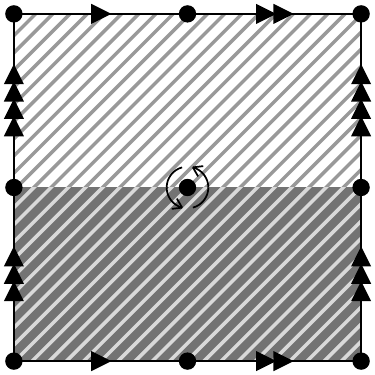}
\begin{picture}(0,0)(104,0)
\put(-14,110){$v_{00}$}
\put(42, 110){$v_{10}$}
\put(99,108){$v_{00}$}
\put(-14,-4){$v_{00}$}
\put(-17,59){$v_{01}$}
\put(42,-6){$v_{10}$}
\put(99,-4){$v_{00}$}
\put(99,58){$v_{01}$}
\end{picture}}
\caption{$(\reals^2-\integers^2)/H$ and a fundamental domain for $G/H$}
\label{modH fig}
\end{minipage}
\qquad
\begin{minipage}{145pt}\centering
\vspace{40pt}
\scalebox{1}{\includegraphics{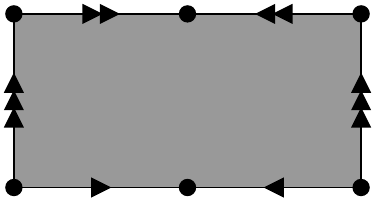}
\begin{picture}(0,0)(104,0)
\put(-14,-4){$v_{00}$}
\put(-14,59){$v_{01}$}
\put(42.5,60){$v_{11}$}
\put(42.5,-6){$v_{10}$}
\put(99,-4){$v_{00}$}
\put(99,58){$v_{01}$}
\end{picture}}
\vspace{5pt}
\caption{$(\reals^2-\integers^2)/G$}
\label{modG fig}
\end{minipage}
\end{figure}

\section{Taggable arcs and tagged arcs}

In this section we show that tagged arcs are specified by, among other data, a \newword{rational slope}, i.e., a rational number or $\infty$.  
Every rational slope has a unique \newword{standard form} which expresses the slope as a fraction $b/a$ 
where $a$ is a nonnegative integer and $b$ is an integer subject to the requirements that $b=1$ if $a=0$ and that $\gcd(a,b)=1$ if $a>0$. 
(By convention, if $a>0$, then $\gcd(a,0)=a$ and $\gcd(a,b)=\gcd(a,-b)$ for $b<0$.)
Consider an integer point $[p,q]\in\integers^2$ and the standard form $b/a$ of a rational slope.
Because $b/a$ is in standard form, the straight line segment connecting $[p,q]$ to $[p+a,q+b]$ intersects no other integer points.
Since any image of this line segment under the action of $G$ is a straight line segment connecting some point $[r,s]$ to the point $[r+a,s+b]$, and such line segments intersect at most at shared endpoints, this line projects to a non-self-intersecting curve in the four-punctured sphere with no punctures in its interior.  
Since in particular $a$ and $b$ are not both multiples of~2, the points $[p,q]$ and $[p+a,q+b]$ do not project to the same puncture, so the projected curve is not a loop.
Thus the projected curve is a taggable arc in the sense of Definition~\ref{def: arcs}.

We write $\arc(a,b,\set{v_{pq},v_{p+a, q+b}})$ for this taggable arc, interpreting the subscripts on $v_{pq}$ and $v_{p+a,q+b}$ modulo $2$.
The notation is well-defined because two line segments connecting integer points in $\R^2$ are related by the action of $G$ if and only if they have the same slopes and their endpoint sets agree modulo 2.
The set braces in the notation remind us that $v_{pq}$ and $v_{p+a,q+b}$ appear symmetrically in the construction, because the line segment connecting $[p+a,q+b]$ to $[p+2a,q+2b]$ projects to the same arc in $\mathbb{S}-4$ as the line segment connecting $[p,q]$ to $[p+a,q+b]$.
For any standard form $b/a$, there are exactly two choices for the set $\set{v_{pq},v_{p+a,q+b}}$. 
For example, if $b/a=3/2$, then $\set{v_{pq},v_{p+a,q+b}}$ is either $\set{v_{00},v_{01}}$ or $\set{v_{10},v_{11}}$.

 We use a similar notation for tagged arcs by attaching a superscript $\notch$ to the endpoints $v_{pq}$ and/or $v_{p+a,q+b}$ to indicate notched tagging of the arc at that endpoint.
The four tagged arcs whose underlying taggable arc is $\arc(a,b,\set{v_{pq},v_{rs}})$
are  \[\arc(a,b,\set{v_{pq},v_{rs}}), \, \arc(a,b,\set{v\notch_{pq},v_{rs}}), \, \arc(a,b,\set{v_{pq},v\notch_{rs}}), \, \arc(a,b,\set{v\notch_{pq},v\notch_{rs}}).\] 
Let $A$ be the set of triples $(a,b,E)$ such that $b/a$ is the standard form for a rational slope and $E$ is $\set{v_{pq},v_{rs}}$, $\set{v\notch_{pq},v_{rs}}$, $\set{v_{pq},v\notch_{rs}}$, or $\set{v\notch_{pq},v\notch_{rs}}$ with $[p,q]+[a,b]=[r,s]$ modulo 2. 

The following proposition is the main result of this section.
\begin{proposition}\label{prop: tagged arcs}
The map $(a,b,E)\mapsto\arc(a,b,E)$ is a bijection from $A$ to the set of tagged arcs in the four-punctured sphere. 
\end{proposition}

To prove Proposition~\ref{prop: tagged arcs}, we relate taggable arcs in $\Sp-4$ to taggable arcs in $\T-1$ that are symmetric with respect to the action of the unique non-identity element $\sigma$ in $G/H$.
Since the puncture $v_{00}$ is fixed by $\sigma$, the action of $\sigma$ on taggable arcs in $\T-1$ is well-defined.
Taggable arcs in $\T-1$ were classified in \cite[Proposition~4.1]{unitorus}.
In that paper, $\T-1$ is realized by modding out the integer-punctured plane by the group of integer translations.
In this paper, we realize $\T-1$ as $(\R^2-(2\Z)^2)/H$.
Informally, $\T-1$ is obtained from $\T-4$ by ``filling in'' the punctures $v_{01}$, $v_{11}$, and $v_{10}$, leaving only the puncture $v_{00}$.
The following proposition is a straightforward translation of \cite[Proposition~4.1]{unitorus}, from a statement about the once-punctured torus given by $\R^2-\Z^2$ modulo integer translations to a statement about the once-punctured torus given by $(\R^2-(2\Z)^2)/H$.
(Note that in the case of $\T-1$, the definitions of (ordinary) arcs and taggable arcs coincide.)

\begin{proposition}\label{unitorus taggable arcs}  
The taggable arcs in the once-punctured torus $(\R^2-(2\Z)^2)/H$ are the images, under the covering map, of straight line segments connecting $[0,0]$ to integer points $[2a,2b]$ such that $b/a$ is the standard form of a rational slope.
This map from standard forms of rational slopes to taggable arcs is a bijection.
\end{proposition}
The line segment connecting $[0,0]$ to $[2a,2b]$ is centrally symmetric about its midpoint $[a,b]$.
When we project to $\T-1$, this central symmetry becomes symmetry with respect to the action of $\sigma$. 

\begin{definition}\label{sigma-symmetric}
A taggable arc $\alpha$ in the once-punctured torus is \newword{$\sigma$-symmetric} if there is a parametrization of $\alpha$ as a map from $[0,1]$ to the once-punctured torus such that for each $s\in [0,1]$, $\alpha(s) = \sigma(\alpha(1-s))$.
We will call such a parametrization a \newword{$\sigma$-symmetric parametrization}.
We will use the notation $f_t(s)$ for an isotopy, parametrized by $(s,t)\in[0,1]\times [0,1]$, between the curves $f_0$ and $f_1$.
A \newword{$\sigma$-symmetric isotopy} is an isotopy between two $\sigma$-symmetric arcs in the once-punctured torus such that for each $(s,t) \in [0,1]\times [0,1]$, $f_t(s)=\sigma\left(f_t(1-s)\right)$.
That is, each $f_t$ is a $\sigma$-symmetric curve with a $\sigma$-symmetric parametrization $f_t(s)$.
Two $\sigma$-symmetric arcs are \newword{$\sigma$-symmetrically isotopic} if there is a $\sigma$-symmetric isotopy between them.
\end{definition}

The following proposition is immediate.
\begin{proposition}\label{centrally symmetric}
A taggable arc in $\T-1$ is $\sigma$-symmetric if and only if each of its lifts to $\R^2-(2\Z)^2$ (or equivalently one of its lifts) is centrally symmetric with respect to some integer point.
\end{proposition}
To relate taggable arcs in $\Sp-4$ to $\sigma$-symmetric arcs in $\T-1$, we will construct a map $\phi$ which sends a curve in $\Sp-4$ with one endpoint at $v_{00}$, satisfying the conditions of Definition~\ref{def: arcs}, to a $\sigma$-symmetric taggable arc in $\T-1$.
Given such a curve $\alpha$, let $\alpha_1$ and $\alpha_2$ be its two lifts to $\T-4$.
Since punctures lift to punctures and the covering map is a local homeomorphism, both $\alpha_1$ and $\alpha_2$ are taggable arcs in $\T-4$.
We concatenate $\alpha_1$ with the reverse of $\alpha_2$ to form a loop $\overline\alpha$ in $\T-1$ with both endpoints at $v_{00}$.  
Since $\alpha_1$ and $\alpha_2$  meet the conditions of Definition~\ref{def: arcs} in $\T-4$, $\overline{\alpha}$ has no self-intersections, no interior punctures, and is not contractible to the puncture in $\T-1$.
Since no arc based at $v_{00}$ can bound a once-punctured monogon in $\T-1$, we conclude that $\overline{\alpha}$ meets the conditions of Definition~\ref{def: arcs}.   
Since the two lifts $\alpha_1$ and $\alpha_2$ in $\T-4$ are related by the action of $\sigma$, the arc $\overline{\alpha}$ is $\sigma$-symmetric.

\begin{proposition}\label{bij to symmetric arcs}
The map $\phi: \alpha\mapsto \overline{\alpha}$ is a bijection between curves with an endpoint at $v_{00}$ meeting the conditions of Definition~\ref{def: arcs} in the four-punctured sphere and $\sigma$-symmetric taggable arcs in the once-punctured torus.
\end{proposition}
\begin{proof}
We construct a map $\psi$ that is an inverse for $\phi$.
Given a $\sigma$-symmetric curve $\alpha$ in $\T-1$ meeting the conditions of Definition~\ref{def: arcs} (in $\T-1$), choose a $\sigma$-symmetric parametrization of $\alpha$. 
Restrict $\alpha$ to $\T-4$, and observe that its restriction has no points of self-intersection.
Since the punctures are the only $\sigma$-fixed points, $\alpha(1/2)$ is a puncture $v' \neq v_{00}$.
Moreover, $\alpha$ intersects no punctures except $v_{00}$ (at $s=0$ and $s=1$) and $v'$ (at s=1/2), because otherwise its $\sigma$-symmetry implies that it has a point of self-intersection, a contradiction.  
The restriction of $\alpha$ to the interval $[0,1/2]$ is a curve in $\T-4$ that meets the conditions of Definition~\ref{def: arcs} (in $\T-4$), and $\psi(\alpha)$ is the projection of this curve to $\Sp-4$.
Note that $\psi$ does not depend on the direction of the $\sigma$-symmetric parametrization of $\alpha$ since the restriction of $\alpha$ to $[0,1/2]$ projects to the same curve in $\Sp-4$ (but parametrized in the opposite direction) as the restriction of $\alpha$ to $[1/2,1]$.  
\end{proof}
Proposition~\ref{bij to symmetric arcs} was about ``curves meeting the conditions of Definition~\ref{def: arcs}'' rather than simply about ``taggable arcs,'' because taggable arcs are considered up to isotopy, whereas Proposition~\ref{bij to symmetric arcs} is a statement about individual curves.
We now show that we can consider taggable arcs (up to isotopy as usual) by considering their images under $\phi$ up to $\sigma$-symmetric isotopy.
\begin{proposition}\label{sigma isotopic}
Two curves in the four-punctured sphere meeting the conditions of Definition~\ref{def: arcs}, both with an endpoint at $v_{00}$, are isotopic if and only if their images under $\phi$ are $\sigma$-symmetrically isotopic in the once-punctured torus.
\end{proposition}
\begin{proof}
Fix an isotopy $f:[0,1]\times[0,1]\to \Sp-4$ between curves $\alpha$ and $\beta$ with endpoints at $v_{00}$ and $v''$ (where $v'' \in  \M\setminus \{v_{00}\}$) that satisfy the conditions of Definition~\ref{def: arcs}. 
We restrict each $f_t$ to the open interval $(0,1)$ and appeal to the homotopy lifting property to lift this restriction of $f$, obtaining an isotopy $g: (0,1)\times [0,1] \to \T-4$.
Since for each $t\in [0,1]$, $g_t$ limits to $v_{00}$ as $s$ goes to $0$ and limits to $v''$ as $s$ goes to $1$, we can continuously extend $g$ so that it is defined on $[0,1]\times [0,1]$.
By acting on $g$ by $\sigma$, we obtain another lift $h$.   
As in our construction of the map $\phi$ above, we concatenate $g$ with the reverse of $h$.
We denote the result by $\overline{f}$, and observe that for each $t$, $\overline{f_t}$ is a $\sigma$-symmetric arc in $\T-1$ that can be given a $\sigma$-symmetric parametrization so that the projection of $\overline{f_t}(s)$ to $\Sp-4$ is $f_t(2s)$. 
The map $(s,t)\mapsto\overline{f_t}(s)$ is an isotopy from $\overline{\alpha}$ to $\overline{\beta}$.

Suppose that $f:[0,1]\times [0,1]\to \T-1$ is a $\sigma$-symmetric isotopy between $\sigma$-symmetric curves meeting the conditions of Definition~\ref{def: arcs}.
For each $t\in [0,1]$ and $s\le 1/2$, the composition of $f_t$ and the covering map from $\T-1$ to $\Sp-4$ is precisely the curve $\psi(f_t)$, which satisfies Definition~\ref{def: arcs}.
Thus we obtain an isotopy between $\psi(f_0)$ and $\psi(f_1)$ in $\Sp-4$. 
\end{proof}

Now we have reduced the classification of taggable arcs in $\Sp-4$ to the classification of $\sigma$-symmetric isotopy classes of $\sigma$-symmetric taggable arcs in $\T-1$.
The next result shows that $\sigma$-symmetric arcs are $\sigma$-symmetrically isotopic to the projection of a straight-line segment connecting even integer points.

\begin{proposition}\label{isotopic = sigma isotopic}
Each $\sigma$-symmetric curve $\alpha$ in $\T-1$ satisfying the conditions of Definition~\ref{def: arcs} is $\sigma$-symmetrically isotopic to the projection of a straight line segment with rational slope $b/a$ in standard form connecting the integer points $[0,0]$ and $[2a,2b]$.
\end{proposition}

\begin{proof}
Fix a $\sigma$-symmetric curve $\alpha$ satisfying the conditions of Definition~\ref{def: arcs} in $\T-1$.
By Proposition~\ref{unitorus taggable arcs}, $\alpha$ is isotopic to the projection of a straight line segment~$\lambda$ connecting $[0,0]$ to $[2a,2b]$ with rational slope $b/a$ (in standard form).
 Denote by~$\alpha'$ the lift of $\alpha$ to a curve in $\R^2-(2\Z)^2$ with endpoints $[0,0]$ and $[2a,2b]$.

By Proposition~\ref{centrally symmetric}, both of the curves $\lambda$ and $\alpha'$ are centrally symmetric about $[a,b]$. 
We will construct an isotopy between $\lambda$ and $\alpha'$ that is centrally symmetric, and compose this with the covering map to obtain a $\sigma$-symmetric isotopy in $\T-1$.  
The construction is essentially the same as the construction in the proof of \cite[Proposition~4.1]{unitorus} (our Proposition~\ref{unitorus taggable arcs}).
We begin by revisiting the steps in that construction, and then we argue that there is a way to make these deformations in a centrally symmetric manner.
As in \cite[Proposition~4.1]{unitorus}, we assert that our deformations may be made so that $\alpha'$ remains disjoint from all of its even-integer translates.  
Thus, the resulting isotopy projects to an isotopy of taggable arcs in $\T-1$. 

We consider the $2\times 2$ unit squares with even-integer vertices that $\alpha'$ visits between $[0,0]$ and $[2a,2b]$.
If $\alpha'$ does not visit any square, it forms an edge of the square containing $[0,0]$ and $[2,2]$. 
In that case, it is a straight line segment, with rational slope either zero or infinity, so it must be equal to $\lambda$. 
Thus we assume that $\alpha'$ enters a square. 
Up to rotational symmetry, we can assume that $b/a$ is non-negative and finite.
First we minimize the number of squares that $\alpha'$ visits. 
If $\alpha'$ enters and exits a square out of the same edge, then we can deform it so that it never enters the square at all.

In the case where $b/a$ is zero, after making the deformations above, $\alpha'$ intersects the interior of the line segment connecting $[0,0]$ to $[2,0]$ in at most one point and passes through at most two squares: the square containing $[1,1]$ and the square containing $[1,-1]$.
In particular, $\alpha'$ and $\lambda$ are not separated by any even integer points.

Suppose that $b/a$ is strictly positive.
Then the deformed $\alpha'$ exits the square containing $[0,0]$ through an edge that does not contain $[0,0]$ as a vertex, and the first square that $\alpha'$ visits is the square containing $[1,1]$.
Since $\alpha'$ is disjoint from its even-integer translates, it exits out of the top or right side of each square it visits, and ends at the top-right vertex of the last square it enters.

Next we ``straighten'' $\alpha'$ to a straight line segment in the interior of each square it visits (still assuming $b/a>0$) as follows. 
We have deformed $\alpha'$ so that for each square it passes through, it intersects exactly two edges, and it intersects each edge once.
Consider the first square that our curve visits.
Let $p_0$ be the point where the curve intersects an edge of the square upon entering (in this case $[0,0]$), and let $p_1$ be the point where it intersects an edge upon leaving.  
We alter $\alpha'$ by deforming its intersection with this square to the straight line segment connecting $p_0$ to $p_1$. 
We move $p_1$ to avoid intersecting the interior puncture of the square, if necessary, and so that $\alpha'$ is monotonically increasing in that square.
We do this for each each square that $\alpha'$ visits.
The final result is a monotonic polygonal path in $\R^2-(2\Z)^2$. 

We have already remarked that when $b/a$ is zero, $\alpha'$ and $\lambda$ are not separated by any even integer points.
Thus, there is an isotopy taking $\alpha'$ to $\lambda$.
In the proof of \cite[Proposition~4.1]{unitorus} (our Proposition~\ref{unitorus taggable arcs}), it is shown that for $b/a$ positive, $\alpha'$ and $\lambda$ are not separated by any even integer points by showing that the sequence of top and right edges that $\alpha'$ passes through is the same as the corresponding sequence given by $\lambda$.
Thus, the map that sends each point on $\alpha'$ along a straight line to the closest point to it on $\lambda$ is an isotopy in $\R^2-(2\Z)^2$. 
This completes our construction of an isotopy between $\alpha'$ and $\lambda$.

We now argue that deformations taking $\alpha'$ to $\lambda$ described in the previous paragraphs can be made centrally symmetrically. 
Since $\alpha'$ is centrally symmetric about the point $[a,b]$, any deformations we need to make as it passes through $2\times 2$ unit squares (with even-integer vertices) between $[0,0]$ and $[a,b]$ also need to be made for its path between $[a,b]$ and $[2a,2b]$.
Specifically, any deformation that we make to $\alpha'$ between $[0,0]$ and $[a,b]$ so that it enters each square once, will also be made symmetrically to its image between $[a,b]$ and $[2a,2b]$.  
Thus, the intersections of the curve with edges of squares are situated centrally symmetrically about the point $[a,b]$.
When $b/a$ is zero, this proves that the isotopy between $\alpha'$ and $\lambda$ projects to a $\sigma$-symmetric isotopy in $\T-1$, as desired.
When $b/a$ is positive, this ensures that straightening the resulting curve also preserves central symmetry.
Thus we obtain a centrally symmetric isotopy between $\alpha'$ and a monotonic centrally symmetric polygonal path.
Finally, the isotopy that sends a point on this polygonal path to the closest point to it on $\lambda$ is centrally symmetric because both the polygonal path and $\lambda$ are centrally symmetric.
Thus, the isotopy between $\alpha'$ and $\lambda$ can be made centrally symmetric and disjoint from its even-integer translates.
We conclude its projection to $\T-1$ is a $\sigma$-symmetric isotopy between $\alpha$ and the projection of $\lambda$.  
\end{proof}

We can now prove our main result.
\begin{proof}[Proof of Proposition~\ref{prop: tagged arcs}]
Since we can tag taggable arcs arbitrarily and relabel vertices if necessary, it is enough to show that the map which sends $(a,b,\{v_{00},v_{ab}\})$ to $\arc(a,b,\{v_{00},v_{ab}\})$ is a bijection from triples $(a,b,\set{v_{00},v_{ab}})$ such that $b/a$ is the standard form of a rational slope to taggable arcs in $\Sp-4$ having an endpoint at~$v_{00}$. 
Fix a curve $\alpha$ in $\Sp-4$ with an endpoint at $v_{00}$ satisfying Definition~\ref{def: arcs}.
By Proposition~\ref{isotopic = sigma isotopic}, $\phi(\alpha)$ is $\sigma$-symmetrically isotopic to the projection of a straight line segment with rational slope $b/a$ in standard form connecting $[0,0]$ to $[2a,2b]$. 
Then by Proposition~\ref{sigma isotopic}, $\alpha$ is isotopic to the projection of the straight line connecting $[0,0]$ to $[a,b]$.

Moreover, Proposition~\ref{unitorus taggable arcs} implies that two straight line segments connecting $[0,0]$ to $[2a,2b]$ and $[0,0]$ to $[2c,2d]$, with distinct rational slopes $b/a$ and $d/c$ (in standard form), project to arcs that are not even isotopic (let alone $\sigma$-symmetrically isotopic) in $\T-1$, so by Proposition~\ref{sigma isotopic}, we conclude that the map is injective.
\end{proof}


\section{Compatibility of tagged arcs}\label{arc sec} 
Proposition~\ref{prop: tagged arcs} parametrizes tagged arcs by rational slopes, pairs of endpoints, and taggings.
In this section, we describe compatibility of tagged arcs in terms of that parametrization.

\begin{proposition}\label{prop: tagged compat}
Two tagged arcs in the four-punctured sphere whose underlying taggable arcs coincide are compatible if and only if their tagging agrees at exactly one of their endpoints.
Two tagged arcs $\arc(a,b,E)$ and $\arc(c,d,F)$ with distinct underlying taggable arcs are compatible if and only if their taggings agree at each shared endpoint and $|ad-bc|$ equals the number of shared endpoints ($0$, $1$, or $2$).
\end{proposition}

Before we prove Proposition~\ref{prop: tagged compat}, we introduce some terminology suggested by the proposition.
We call two distinct tagged arcs a \newword{coinciding pair} if their underlying taggable arcs coincide and their taggings make them compatible as described in Proposition~\ref{prop: tagged compat}.  
Rational slopes $b/a$ and $d/c$ with $|ad-bc|=1$ are sometimes called Farey pairs, or Farey neighbors, because they are adjacent in some Farey sequence.  
(See for example  \cite[Section~4.5]{Graham} or \cite[Chapter 1]{Hatcher2}.)
Accordingly, we call \emph{non-coinciding} compatible pairs $\arc(a,b,E)$ and $\arc(c,d,F)$ \newword{Farey-0 pairs} (or \newword{parallel pairs}), \newword{Farey-1 pairs}, or \newword{Farey-2 pairs} according to their number of shared endpoints. 
We overload the terms by also referring to ``Farey-1 (or Farey-2) pairs of slopes'' and ``Farey-1 (or Farey-2) compatible pairs of slopes.''    
A \newword{Farey-1 triple} is a collection of three slopes, each pair of which is a Farey-1 compatible pair.

\begin{lemma}\label{lem: plane compatibility}  
Two distinct taggable arcs $\arc(a,b,\bar{E})$ and $\arc(c,d,\bar{F})$ are compatible if and only if the following condition holds: Given any line in the plane with slope $b/a$ passing through a point in the preimage of $\bar{E}$ (under the covering map from $\R^2-\Z^2$ to $\Sp-4$) and any line with slope $d/c$ passing through a point in the preimage of $\bar{F}$, either the two lines are parallel or their intersection is an integer point. 
\end{lemma}

\begin{proof}
By definition, $\arc(a,b,\bar{E})$ and $\arc(c,d,\bar{F})$ are compatible if and only if, up to isotopy relative to $\M$, they intersect precisely at their shared endpoints. Lifting this characterization to the $\integers^2$-punctured plane, $\arc(a,b,\bar{E})$ and $\arc(c,d,\bar{F})$ are compatible if and only if their respective fibers under the covering map intersect precisely at the preimage of these shared endpoints. 
By Proposition~\ref{prop: tagged arcs}, these fibers are the set of lines of slope $b/a$ passing through points in the preimage of $\bar{E}$, and the set of lines of slope $d/c$ passing through points in the preimage of $\bar{F}$, respectively. The integer points on these lines are exactly the preimage of endpoints of the two arcs. The lemma follows. 
\end{proof}

\begin{proof}[Proof of Proposition~\ref{prop: tagged compat}]
It suffices to prove that distinct taggable arcs $\arc(a,b,\bar{E})$ and $\arc(c,d,\bar{F})$ are compatible if and only if $|ad-bc| = |\bar{E} \cap \bar{F}|$, the number of shared endpoints. We will apply Lemma~\ref{lem: plane compatibility}: let $\ell_1$ and $\ell_2$ be arbitrary lines of slope $b/a$ and $d/c$, respectively, passing through respective integer points $[m_1, p_1]$ in the preimage of $\bar{E}$ and $[m_2, p_2]$ in the preimage of $\bar{F}$.

The zero-shared-endpoints case is dispensed with easily. If $\bar{E} \cap \bar{F} = \emptyset$, then $\arc(a,b,\bar{E})$ and $\arc(c,d,\bar{F})$ are compatible if and only if $\ell_1 \cap \ell_2 = \emptyset$, or in other words, $\ell_1$ and $\ell_2$ are parallel. This occurs if and only if $b/a=d/c$, or equivalently, $|ad-bc|=0$.

The one- and two-shared-endpoint cases require more care. In these cases, $b/a \neq d/c$. Thus $\ell_1$ and $\ell_2$ intersect at {\em some} point $[x,y]$, and we must show that $|ad-bc|$ equals the number of shared endpoints if and only if $[x,y] \in \integers^2$. We begin by proving the forward implication. 
A simple computation yields
\begin{equation}
\label{eq: intersection}
x=\frac{adm_2-bcm_1+ac(p_1-p_2)}{ad-bc}, \ y = \frac{adp_1-bcp_2 + bd(m_2-m_1)}{ad-bc}
\end{equation} 
If $|ad-bc| = |\bar{E} \cap \bar{F}| =1$, then Equation~\eqref{eq: intersection} implies that $[x,y]$ is an integer point. 
If $|ad-bc| = |\bar{E} \cap \bar{F}| = 2$, then either $v_{m_1p_1}=v_{m_2p_2}$ or $v_{m_1p_1}=v_{m_2+c,p_2+d}$. We assume without loss of generality that the former holds, since $\ell_2$ also passes through the integer point $[m_2', p_2'] := [m_2+c, p_2+d]$ and we can rewrite Equation~\eqref{eq: intersection} accordingly. 
Since the indices on punctures are interpreted modulo 2, we have $m_1 \equiv m_2 \mod 2$ and $p_1 \equiv p_2 \mod 2$. From there, it is easy to use Equation~\eqref{eq: intersection} to show $[x,y] \in \Z^2$.

Conversely, suppose $[x,y] \in \integers^2$ so that the two arcs are compatible. We show that $|ad-bc|$ equals the number of shared endpoints using Pick's Theorem: 
The area of a polygon whose vertices are integer lattice points is given by $I + B/2 -1$ where~$I$ is the number of lattice points in the polygon's interior and $B$ is the number of lattice points on the polygon's boundary.  
For a proof of Pick's Theorem, see for example \cite[Theorem~13.51]{Coxeter}.
We cover the one-shared-endpoint case in depth, and then provide an outline of the proof of the two-shared-endpoint case, which is similar.

Given $\bar{E} \cap \bar{F} = \set{v_{xy}}$,  consider the parallelogram $P$ in $\R^2$ with vertices at $[x,y]$, $[x+2a,x+2b]$, $[x+2c, y+2d]$, and $[x+2a+2c, y+2b+2d]$ as depicted in the left picture of Figure~\ref{fig: parallelograms}.   
We make the following observations:
\begin{itemize}
\item Each vertex of $P$ is a lift of the puncture $v_{xy}$.
\item Each side of $P$ with slope $b/a$ is composed of two lifts of $\arc(a,b,\bar{E})$ and each side with slope $d/c$ is composed of two lifts of $\arc(c,d,\bar{F})$.
\item Since taggable arcs have no punctures in their interior, the lifts of $\arc(a,b,\bar{E})$ and $\arc(c,d,\bar{F})$ cannot pass through integer points. Thus $P$ has precisely eight integer points on its boundary, namely: 
\begin{itemize}
\item the four corner vertices, each a lift of the puncture $v_{xy}$,
\item midpoints $[x+a, y+b]$ and $[x+a+2c, x+b+2d]$, both lifts of the distinct puncture $v_{x+a,y+b}$, and 
\item midpoints $[x+c,y+d]$ and $[x+2a+c, y+2b+d]$, both lifts of $v_{x+c,y+d}$, a third distinct puncture.
\end{itemize}
\item $P$ contains at least one integer point in its interior, namely the center $[x+a+c, y+b+d]$, a lift of the fourth puncture.
\item $P$ has area $4|ad-bc|$.
\end{itemize}
By Pick's Theorem, the total number of integer points in the interior of $P$ is given by $4|ad-bc|-3$. 
Therefore, to conclude that $|ad-bc|=1$, we need to show that the center $[x+a+c,y+b+d]$ is the only integer point in the interior of $P$.

To that end, let $[x',y']$ be any other point in the interior of $P$. We show that $[x',y']$ is not an integer point by proving that it cannot be the lift of a puncture, analyzing two cases. We define the \newword{midsegments} of $P$ to be the two line segments connecting the midpoints of opposite sides of $P$.
First, assume $[x',y']$ does not lie on either midsegment. Then the line of slope $b/a$ passing through $[x',y']$ intersects the boundary of $P$ at non-integer points along $(d/c)$-sloped sides of $P$. Since $\arc(a,b,\bar{E})$ and $\arc(c,d,\bar{F})$ are compatible, Lemma~\ref{lem: plane compatibility} implies that $[x',y']$ is a lift of neither $v_{xy}$ nor $v_{x+a,y+b}$. 
Similarly, considering the line of slope $d/c$ through $[x',y']$, we see that $[x',y']$ is also not a lift of $v_{x+c,y+d}$. 
Finally, we show $[x',y']$ is not a lift of the fourth puncture $v_{x+a+c, y+b+d}$. 
Since lifts of the same puncture are even-integer translates of one another, any line segment connecting two such lifts must pass through an integer point which is the lift of a distinct puncture. 
Thus if $[x',y']$ were a lift of $v_{x+a+c, y+b+d}$, then the line segment connecting it to the center $[x+a+c, y+b+d]$ of $P$ would necessarily pass through a lift of one of the other three punctures: $v_{xy}, v_{x+a,y+b}$, or $v_{x+c,y+d}$. But every point in the interior of the line segment connecting $[x',y']$ to the center of $P$ is an interior point of $P$ which does not lie on a midsegment, and we argued above that no such point can be a lift of $v_{xy}, v_{x+a,y+b}$, or $v_{x+c,y+d}$.

 Now, assume that $[x', y']$ lies on the midsegment connecting $[x+a, y+b]$ to $[x+a+2c, x+b+2d]$. The line of slope $b/a$ passing through $[x',y']$ intersects the boundary of $P$ at non-integer points along the $(d/c)$-sloped sides of $P$, and hence $[x',y']$ cannot be a lift of either $v_{xy}$ or $v_{x+a, y+b}$ as discussed above. If $[x', y']$ were a lift of $v_{x+c,y+d}$, then the line segment connecting $[x',y']$ to $[x+c, y+d]$, a lift of the same puncture, would pass through an integer point. But every point on the interior of this line segment is a non-midsegment interior point, and we've already eliminated the possibility of any such point being integral. We can now rule out the possibility of $[x',y']$ being a lift of $v_{x+a+c,y+b+d}$ by the same reasoning used above in the non-midsegment case. The argument for when $[x', y']$ lies on the other midsegment of $P$ is similar, thereby completing our one-shared-endpoint proof.

The proof when $\arc(a,b,\bar{E})$ and $\arc(c,d,\bar{F})$ share two endpoints is analogous; in this case, we consider the parallelogram $P'$ with vertices at $[x,y]$, $[x+a,y+b]$, $[x+c, y+d]$, and $[x+a+c, y+b+d]$, as shown in the right picture of Figure~\ref{fig: parallelograms}. 
Observe that each side of slope $b/a$ is a single lift of $\arc(a,b,\bar{E})$ while each side of slope $d/c$ is a single lift of $\arc(c,d,\bar{F})$, so that $v_{x+a,y+b} = v_{x+c, y+d}$, $v_{xy}=v_{x+a+c, y+b+d}$ and $P'$ has precisely 4 integer points on its boundary. 
It follows by Pick's Theorem that the total number of points in the interior of $P'$ is given by $|ad-bc|-1$. 
Since $[x,y]$ and $[x+a+c, y+b+d]$ are equal modulo 2, there is at least one integer point in the interior of $P'$, namely, the center $[x+\frac{a+c}{2}, y+\frac{b+d}{2}]$. 
We show that this is the only interior integer point by ruling out the possibility of any other interior lifts of punctures using the same techniques as in the one-shared-endpoint case, and thereby conclude that $|ad-bc|=2$.
\end{proof}

\begin{figure}
\centering
  \includegraphics{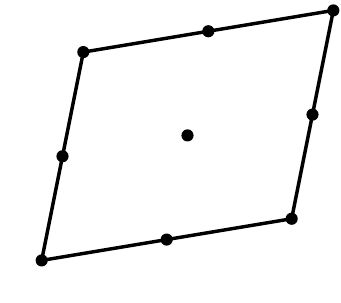}\begin{picture}(0,0)(86,-12)
  \put(-12,-10){\small$[x,y]$}
  \put(-49,28){\small$[x+c,y+d]$}
  \put(-51,58){\small$[x+2c,y+2d]$}
  \put(25,-5){\small$[x+a,y+b]$}
  \put(75,10){\small$[x+2a,y+2b]$}
  \end{picture}
   \qquad
  \qquad
  \qquad
  \qquad
  \quad
  \includegraphics{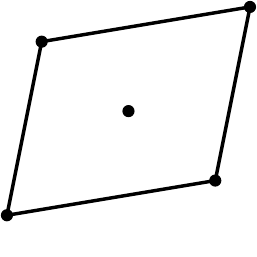}\begin{picture}(0,0)(72,-12)
  \put(-12,-10){\small$[x,y]$}
  \put(-44,49){\small$[x+c,y+d]$}
  \put(63,7){\small$[x+a,y+b]$}
  \end{picture}
\caption{Parallelograms in $\reals^2$ used in the proof of Proposition~\ref{prop: tagged compat}}
\label{fig: parallelograms}
\end{figure}


\section{Tagged triangulations}\label{tagged tri}

With a description of compatibility of tagged arcs in hand, we describe the tagged triangulations of the four-punctured sphere.
By \cite[Theorem~7.9]{fst}, each tagged triangulation has the same number of arcs, and we have seen in Figure~\ref{tri and mat} that this number is $6$.
We illustrate the six combinatorial types of tagged triangulations in Table~\ref{table:triangulations}.
These pictures are drawn, not in the sphere illustrated in Figure~\ref{modG fig}, but rather in a more conventional sphere: the plane of the page compactified with a point at infinity.   
In the right column of Table~\ref{table:triangulations}, $\M$ is the set $\set{v_{00},v_{01},v_{10},v_{11}}$ of punctures of the four-punctured sphere.
We use the notation $[i,j]<[k,l] \mod 2$ to refer to the total order $[0,0] <[0,1] <[1,0] <[ 1,1]$ on $(\integers/2\integers)^2$.  

\begin{proposition}\label{all tri}
The tagged triangulations of the four-punctured sphere are shown in Table \ref{table:triangulations}.  The left column of the table shows the combinatorial type of each triangulation and the right column describes the data which uniquely specifies each triangulation of that type. 
\end{proposition}

\begin{table}[h!]\setlength{\tabcolsep}{.5cm}
	\begin{center}
	\begin{tabular}{ c   p{6cm}  }
	\toprule
	Tagged triangulation  &  Data to determine the triangulation \\ \hline
     \raisebox{-\totalheight}{\scalebox{0.9}{\includegraphics{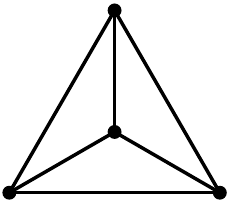}\begin{picture}(0,0)(33,-20)\put(-55,37){\textrm{I}.}\put(-47, -16){$v_{00}$}\put(-22, 10){$p$}\put(-15, -2){$q$}\put(-3, -16){$r$}\end{picture}}}
      & \begin{itemize}[topsep=0pt, leftmargin=10pt, labelindent=0pt, itemindent=0pt]
      \item Farey-1 triple $p,q,r \in \mathbb{Q}\cup \infty$ 
      \item taggings at all vertices
      \end{itemize}
	\\ \bottomrule
	
	 \raisebox{-\totalheight}{\scalebox{0.9}{\includegraphics{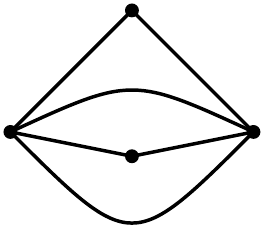}\begin{picture}(0,0)(33,-20)\put(-60, 45){\textrm{II}.}\put(-50,5){$v$}\put(-13, 22){$p$}\put(-13, -14){$q$}\end{picture}}}
	      & \begin{itemize}[topsep=0pt, leftmargin=10pt, labelindent=0pt, itemindent=0pt]
      \item Farey-2 pair $p=b/a,q \in \mathbb{Q}\cup \infty$
      \item $v=v_{ij} \in \M$ such that \newline$[i,j] < [i+a,j+b] \mod 2$ 
      \item taggings at all vertices
      \end{itemize}
	\\ \bottomrule
	
	 \raisebox{-\totalheight}{\scalebox{0.9}{\includegraphics{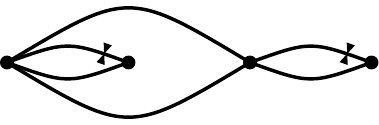}\begin{picture}(-10,-10)(33,-20)\put(-82, 15){\textrm{III}.} \put(-85,-5){$v$}\put(-55, 15){$p$}\put(-55,-22){$q$}\put(-35,-4){$v'$}\end{picture}}}
      & \begin{itemize}[topsep=0pt, leftmargin=10pt, labelindent=0pt, itemindent=0pt]
      \item Farey-2 pair $p=b/a,q \in \mathbb{Q}\cup \infty$
      \item $v=v_{ij} \in \M$ such that \newline $[i,j] < [i+a,j+b] \mod 2$ 
      \item $v' \in \M \setminus \{v_{ij},v_{i+a,j+b}\}$
      \item taggings at $v_{ij},v_{i+a,j+b}$
      \end{itemize}
	\\ \bottomrule
	
	 \raisebox{-\totalheight}{\scalebox{0.9}{\includegraphics{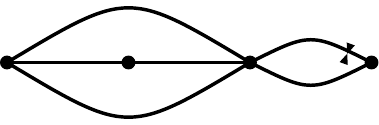}\begin{picture}(0,0)(33,-20)\put(-77,15){\textrm{IV}.}\put(-5,2){$v$}\put(-25, 14){$p$}\put(-25,-19){$q$}\put(35,-3){$v'$}\end{picture}}}
      & \begin{itemize}[topsep=0pt, leftmargin=10pt, labelindent=0pt, itemindent=0pt]
      \item Farey-2 pair $p=b/a,q \in \mathbb{Q}\cup \infty$
      \item $v=v_{ij} \in \M$  
      \item $v' \in \M \setminus \{v_{ij},v_{i+a,j+b}\}$
      \item taggings at $ \M \setminus \{v'\}$
      \end{itemize}
	\\ \bottomrule
	
	\raisebox{-\totalheight}{\scalebox{0.9}{\includegraphics{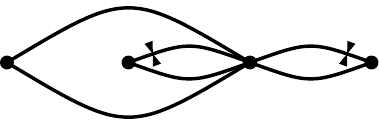}\begin{picture}(0,0)(33,-20)\put(-77,15){\textrm{V}.}\put(-5,2){$v$}\put(-25, 14){$p$}\put(-25,-19){$q$}\end{picture}}}
      & \begin{itemize}[topsep=0pt, leftmargin=10pt, labelindent=0pt, itemindent=0pt]
      \item Farey-2 pair $p=b/a,q \in \mathbb{Q}\cup \infty$
      \item $v =v_{ij}\in \M$  
      \item taggings at $v_{ij}, v_{i+a,j+b}$
      \end{itemize}
	\\ \bottomrule
	
	\raisebox{-\totalheight}{\scalebox{0.9}{\includegraphics{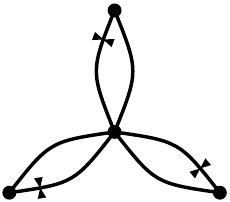}\begin{picture}(0,0)(33,-20)\put(-55,35){\textrm{VI}.}\put(-3,-10){$v$}\put(-22,-1){$p$}\put(16,0){$q$}\put(-12,15){$r$}\end{picture}}}
      & \begin{itemize}[topsep=0pt, leftmargin=10pt, labelindent=0pt, itemindent=0pt]
      \item Farey-1 triple $p,q,r \in \mathbb{Q}\cup \infty$
      \item $v \in \M$  
      \item tagging at $v$
      \end{itemize}
	\\ \bottomrule
	\end{tabular}
	\caption{Tagged triangulations of the four-punctured sphere.}
\label{table:triangulations}
	\end{center}
	\end{table}
	
The following simple observations will be used repeatedly in the proof of Proposition \ref{all tri}.  The first holds because a Farey-1 pair of arcs share exactly one endpoint.
The second holds because a Farey-2 pair of arcs cut the four-punctured sphere into two once-punctured digons. 

\begin{lemma} \label{lemma: three arcs}
In a tagged triangulation of the four-punctured sphere, at most three arcs incident to a single vertex are pairwise Farey-1 compatible.
\end{lemma}

\begin{lemma}\label{lemma: Farey-2}
Suppose $T$ is a tagged triangulation of the four-punctured sphere.
\begin{itemize}
 \item  $T$ contains at most one Farey-2 pair of arcs.
 \item  If $T$ contains a Farey-2 pair of arcs, then each puncture not incident to the Farey-2 pair is incident to exactly two arcs.
 \item If $T$ contains a Farey-2 pair of arcs, then the slope of any other arc of $T$ is uniquely determined by the endpoints of the arc and the slopes of the Farey-2 pair. 
 \end{itemize}
\end{lemma}

\begin{proof}[Proof of Proposition~\ref{all tri}]

Ignoring taggings, a tagged triangulation is in particular a graph.
We first show that any tagged triangulation of the four-punctured sphere has the degree sequence of a diagram shown in Table \ref{table:triangulations}.
We then examine each possible degree sequence, describing the constraints placed on the triangulation by its degree sequence.  

Since each tagged triangulation $T$ contains 6 arcs, the sum of the degrees is~12.  
If a collection of compatible arcs has some vertex $v$ with degree less than two, then one or two additional compatible arcs can be added to the collection to create a coinciding pair incident to that vertex.  
 Thus, since $T$ is a {\em maximal} collection of compatible arcs, no vertex $v$ of $T$ has degree less than 2.
 Therefore, the degree sequence of each $T$ is in the following list: $(3,3,3,3)$, $ (2,3,3,4)$, $(2,2,4,4)$, $(2,2,3,5)$, $(2,2,2,6)$.  
We will see that, other than $(2,3,3,4)$, each of these sequences occurs in at least one triangulation. 

We now show that no tagged triangulation has degree sequence $(2,3,3,4)$.  
For the sake of contradiction, assume that such a tagged triangulation $T$ exists and let $v$ be the vertex of degree $4$.
By Lemma \ref{lemma: Farey-2}, $T$ contains no Farey-2 pairs of arcs.
At least one endpoint of any coinciding pair of arcs has degree two, so $T$ contains at most one coinciding pair.  
Since Lemma \ref{lemma: three arcs} states that at most three arcs incident to $v$ can be pairwise Farey-1 compatible, there must be a coinciding pair incident to $v$.
The slope of the coinciding pair forms a Farey-1 triple with the slopes of the other two arcs incident to $v$.
Let $v'$ be the other endpoint of the coinciding pair  in the triangulation.
Since the remaining two punctures in $\M \setminus \{v, v'\}$ both have degree three, there must be two edges between them.
However, this is impossible  since the multiple edges form either a coinciding or Farey-2 pair and each of these possibilities has already been eliminated. 
Thus there exists no tagged triangulation with degree sequence $(2,3,3,4)$.   

In constructing a tagged triangulation with a given degree sequence, we will deal with the combinatorics first, and then the tagging.
That is, we first construct a tagged triangulation with unspecified taggings as far as possible.
At some vertices, the tagging is specified as part of the combinatorics:
For each coinciding pair, the taggings of arcs are fixed at the endpoint where taggings of the two arcs disagree.
At all other vertices, we may specify a tagging arbitrarily, either plain or notched, with all incident arcs tagged accordingly at that puncture.

Suppose $T$ is a tagged triangulation with degree sequence $(3,3,3,3)$.  
We will show that $T$ is of the combinatorial type shown in row \textrm{I} of Table \ref{table:triangulations}. 
Since no vertex has degree $2$, Lemma \ref{lemma: Farey-2} implies that $T$ contains no Farey-2 pairs of arcs.  
Additionally,~$T$ contains no coinciding pairs of arcs because the degree of one endpoint of any coinciding pair is two. 
Thus any pair of tagged arcs in $T$ incident to the same puncture form a Farey-1 pair and any pair of tagged arcs sharing no punctures form a Farey-0 pair. 
Selecting the three slopes for the three arcs incident to $v_{00}$ uniquely determines the remaining three arcs of the tagged triangulation since each remaining arc forms a Farey-0 pair with one of the three specified arcs. 
Thus, once the tagging at each puncture has been chosen, the tagged triangulation is uniquely determined by a choice of slopes for these three arcs.
We conclude that there is exactly one tagged triangulation of the first combinatorial type for each choice of vertex taggings and Farey-1 triple.

Now suppose $T$ is a tagged triangulation with degree sequence $(2,2,4,4)$.
First assume that $T$ contains no coinciding pairs of arcs.  
We will show that in this case,~$T$ is of the type shown in row \textrm{II} of Table~\ref{table:triangulations}.
By Lemmas  \ref{lemma: three arcs} and \ref{lemma: Farey-2}, the vertices of degree four are the endpoints of a Farey-2 pair of arcs and $T$ contains no other Farey-2 pairs.
Let $p=b/a$ and $q$ be the slopes of the Farey-2 pair and let $v_{ij}$ and~$v_{i+a,j+b}$ be the endpoints of the Farey-2 arcs, with $[i,j]<[i+a,j+b] \mod 2$.
The remaining two arcs incident to $v_{ij}$ form a Farey-1 triple with each of the arcs in the Farey-2 pair. 
By Lemma \ref{lemma: Farey-2}, the slopes of these two arcs are uniquely determined by $p$ and~$q$. 
Similarly, the slopes of the two arcs incident to $v_{i+a, j+b}$ and not in the Farey-2 pair are determined by $p$ and $q$.
To complete the tagged triangulation, the tagging at each puncture can be chosen arbitrarily. 

Now suppose $T$ is a tagged triangulation with degree sequence $(2,2,4,4)$ and at least one coinciding pair of arcs.  
We quickly reach a contradiction if we assume that a coinciding pair of arcs connects the two vertices of degree two since at most two arcs can connect the remaining two vertices, making it impossible for these vertices to have degree four. 
Thus $T$ has a coinciding pair incident to a vertex $v$ of degree four and a vertex $v'$ of degree two, and any other coinciding pair in $T$ must also connect a vertex of degree four to a vertex of degree two.
Consider now the possibilities for the two remaining arcs incident to $v$.
If these arcs either form a coinciding pair (with opposite endpoint at the other vertex of degree two) or form a Farey-1 triple 
with the slope of the coinciding pair connecting $v$ and $v'$, then $T$ cannot be completed to obtain the desired degree sequence. 
The only remaining possibility is that the final two arcs incident to $v$ connect $v$ to the other vertex of degree four and therefore form a Farey-2 pair.
Each of the Farey-2 slopes forms a Farey-1 pair with the slope of the coinciding pair.
The two arcs of $T$ not incident to $v$ are also not incident to $v'$.  
As a result, each forms a Farey-0 pair with each arc of the coinciding pair incident to $v$ and $v'$ and therefore form a coinciding pair with each other.
Selecting a Farey-2 pair 
$p=b/a$ and $q$, a vertex $v$ such that~$v=v_{ij}$ with $[i,j]<[i+a,j+b] \mod 2$, a vertex $v'\in \M \setminus \{v_{ij},v_{i+a,j+b}\}$ for the endpoint of the coinciding pair incident to $v$, and a tagging of the vertices of degree  four uniquely determines a triangulation of the type shown in row \textrm{III} of Table~\ref{table:triangulations}.

Suppose $T$ is a tagged triangulation with degree sequence $(2,2,3,5)$.  
We will show that this triangulation is of the type illustrated in row \textrm{IV} of Table~\ref{table:triangulations}.
Let~$v$ be the degree-5 vertex.  
First suppose that no Farey-2 pairs of arcs are incident to $v$. 
By Lemma  \ref{lemma: three arcs}, two coinciding pairs are incident to $v$.
Additionally, since $\deg(v)=5$, an arc whose slope forms a Farey-1 pair with the slopes of the coinciding pairs is incident to $v$.
In this case, we reach a contradiction since the configuration cannot be completed to form a tagged triangulation with degree sequence $(2,2,3,5)$.  
Thus there is a Farey-2 pair of arcs incident to $v$ with slopes $p=b/a$ and $q$.
By Lemma~\ref{lemma: Farey-2}, there is also a coinciding pair incident to $v$ whose slope forms a Farey-1 pair with each of the arcs in the Farey-2 pair.
The remaining arc incident to $v$ forms a Farey-1 pair with each of the other arcs incident to~$v$.
By Lemma \ref{lemma: Farey-2} the slopes of the Farey-2 pair determine the slopes of the remaining arcs incident to $v$.
Selecting a vertex $v' \in \M \setminus \{v_{ij},v_{i+a,b+j}\}$ determines which of these slopes will correspond to a coinciding pair of arcs.  
The final arc of the triangulation forms a Farey-0 pair with each arc in the coinciding pair of arcs, and is thus uniquely determined.  
Except at~$v'$, taggings can be arbitrarily assigned.  

Next, suppose $T$ is a tagged triangulation with degree sequence $(2,2,2,6)$.  
Let~$v$ be the degree-6 vertex. 
By Lemma \ref{lemma: Farey-2} either two coinciding pairs and a Farey-2 pair are incident to $v$ or three coinciding pairs are incident to $v$.  
First assume that two coinciding pairs and a Farey-2 pair are incident to $v$.   
Once a Farey-2 pair~$p,q$ has been selected, by Lemma \ref{lemma: Farey-2} the slopes of the two coinciding pairs are uniquely determined.
Selecting a tagging for $v$ completes the description of the triangulation.
This combinatorial type is illustrated in row \textrm{V} of Table \ref{table:triangulations}.

Finally, suppose $T$ is a tagged triangulation with degree sequence $(2,2,2,6)$ such that  three coinciding pairs of arcs are incident to $v\in \M$.  
Once a Farey-1 triple $p, q,r$ is selected for the slopes of the coinciding pairs and a tagging for $v$ is chosen, the tagged triangulation is uniquely determined.
This combinatorial type is illustrated in row \textrm{VI} of Table \ref{table:triangulations}.
\end{proof}


\section{Allowable curves and the rational quasi-lamination fan}\label{curve sec}
We now give the classification of allowable curves in the four-punctured sphere and describe compatibility of allowable curves, relying largely on Propositions~\ref{prop: tagged arcs} and~\ref{prop: tagged compat} by way of the map $\kappa$ defined in Section~\ref{defs sec}.
The following proposition describes the classification and defines the notation $\curve(a,b)$ and $\curve(a,b,E)$ for allowable curves that we will use for the rest of the paper.

\begin{proposition}
\label{prop: allowable curves}
The allowable curves in the four-punctured sphere are as follows:
\begin{itemize}
\item For each rational slope in standard form $b/a$, the projection of the straight line with slope $b/a$ from $\R^2-\Z^2$ to $\Sp-4$ is an allowable closed curve that we denote by $\curve(a,b)$.
Up to isotopy in $\Sp-4$, these images depend only on the slope $b/a$.
Every closed allowable curve in $\Sp-4$ is of this form.
\item For each triple $(a,b,E)$ as in Proposition~\ref{prop: tagged arcs},  $\curve(a,b,E)=\kappa(\arc(a,b,E))$ is allowable and every non-closed allowable curve is of this form.
\end{itemize}
Compatibility of allowable curves in the four-punctured sphere is characterized as follows: 
\begin{itemize}
\item No two distinct closed curves are compatible.
\item A closed curve $\curve(a,b)$ and a non-closed curve $\curve(c,d,E)$ are compatible if and only if $[a,b]=[c,d].$
\item Two distinct non-closed allowable curves that coincide, except for the spiral directions at one or both ends, are compatible if and only if their spiral direction coincides at exactly one end. 
\item Two non-closed allowable curves that do not coincide, even after possibly changing spiral directions, are compatible if and only if their spiral directions agree at each shared spiral point and $|ad-bc|$ equals the number of shared spiral points ($0$, $1$, or $2$).
\end{itemize}
\end{proposition}

\begin{proof}
Any line $\ell$ of slope $b/a$ in $\reals^2 - \integers^2$ projects to a curve $\lambda$ in $\Sp-4$.
This curve is closed because two points on $\ell$ differing by $[2a,2b]$ project to the same point in $\Sp-4$.
To show that $\lambda$ is allowable, we show that it cuts $\Sp-4$ into two twice-punctured discs.
Suppose to the contrary that it cuts $\Sp-4$ into two discs, one of which has three or more punctures.
Then there is a $3$-cycle formed by three distinct taggable arcs, none of which intersect $\lambda$. 
But the three taggable arcs are pairwise Farey-1 compatible, and in particular don't all have the same slope.
The union of all lifts of the three arcs to $\reals^2$ is a collection of lines, not all parallel, so in particular $\ell$ intersects a lift of one of the three arcs, and that is a contradiction.

Conversely, any allowable closed curve $\lambda$ in $\Sp-4$ cuts $\Sp-4$ into two twice-punctured discs.
There are exactly two taggable arcs $\alpha_1$ and $\alpha_2$ that don't intersect~$\lambda$, one in each twice-punctured disc.
The union of all lifts of $\alpha_1$ and $\alpha_2$ to $\reals^2$ is a collection of lines, and since $\alpha_1$ and $\alpha_2$ don't intersect, these lines are all parallel. 
Write $b/a$ for the standard form of the slope of these lines.
Furthermore, the lines cover all integer points in the plane and cut the plane into infinite strips.
Since $\alpha_1$ and $\alpha_2$ don't intersect $\lambda$, each lift of $\lambda$ is contained in one of the strips.
Thus each lift of $\lambda$ is isotopic in $\reals^2-\integers^2$ to a straight line with slope $b/a$. 

Given two lines $\ell$ and $\ell'$ in $\reals^2 - \integers^2$ of slope $b/a$, let $\lambda$ and $\lambda'$ be the allowable closed curves obtained from $\ell$ and $\ell'$ by projection.
The taggable arcs $\alpha_1$ and $\alpha_2$ determined by $\lambda$ in the previous paragraph depended only on the slope of $\ell$.
In particular, $\lambda$ and $\lambda'$ determine the same decomposition of $\reals^2$ into strips.
The group~$G$ (see Section~\ref{covering sec}) acts transitively on the strips, and therefore each strip contains a lift of $\lambda$ and a lift of $\lambda'$.
These two lifts are isotopic in $\reals^2-\integers^2$, and thus $\lambda$ and $\lambda'$ are isotopic in $\Sp-4$.
We have proved the first bullet point of the proposition.

The second bullet point is an immediate consequence of Proposition~\ref{prop: tagged arcs} and the fact that $\kappa$ is a bijection between tagged arcs and non-closed allowable curves.

The assertion that distinct closed allowable curves are incompatible and the assertion on compatibility of closed and non-closed allowable curves follow easily from the lines-and-strips ideas in the paragraphs above.
The assertions about compatibility of non-closed allowable curves is immediate from Proposition~\ref{prop: tagged compat} by way of the bijection $\kappa$.
\end{proof}

\begin{proposition}
\label{prop: maximal curves}
Applying $\kappa$ to any triangulation described in Table~\ref{table:triangulations} results in a maximal collection of pairwise compatible allowable curves on the four-punctured sphere.  
The remaining maximal collections of pairwise compatible allowable curves are obtained by applying $\kappa$ to the tagged arcs in the configuration shown in Table \ref{table:curves}.
Table \ref{table:curves} has the same format as Table \ref{table:triangulations}.  
\end{proposition}

\begin{table}[h!]\setlength{\tabcolsep}{.5cm}
	\begin{center}
	\begin{tabular}{ c   p{6cm}  }
	\toprule
     \raisebox{-\totalheight}{\scalebox{0.9}{\includegraphics{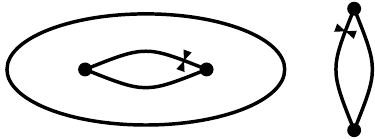}\begin{picture}(0,0)(42,-20)\put(-75,20){\textrm{VII}.}\put(-51, -1){$v$}\put(40, -20){$v'$}\end{picture}}}
      & \begin{itemize}[topsep=0pt, leftmargin=10pt, labelindent=0pt, itemindent=0pt]
      \item slope $p=b/a \in \mathbb{Q}\cup \infty$ 
      \item $v~\in~\{v_{00}, v_{ab}\}$ and $v'\in \M \setminus \{v_{00}, v_{ab}\}$
      \item taggings at $v, v'$
      \end{itemize}
	\\ \bottomrule
	\end{tabular}
	\caption{The unique combinatorial type of a maximal collection of pairwise compatible allowable curves containing a closed curve.  Non-closed curves are represented by tagged arcs.}
\label{table:curves}
	\end{center}
	\end{table}
	
\begin{proof}
Applying $\kappa$ to any triangulation in Table~\ref{table:triangulations}  results in a collection of pairwise compatible non-closed allowable curves containing no closed curves.
Since every closed curve separates the four-punctured sphere into two twice-punctured discs, no closed curve is compatible with this collection.
Thus applying $\kappa$ to each triangulation of Table~\ref{table:triangulations} results in a distinct maximal collection of pairwise compatible allowable curves.
Since $\kappa$ is a bijection, every maximal collection of pairwise allowable curves containing no closed curves arises in this way.

By Proposition \ref{prop: allowable curves}, any other maximal collection of pairwise compatible allowable curves contains a single closed curve.
Let $\lambda$ be a closed curve with slope $b/a$.
Proposition~\ref{prop: allowable curves} states that any allowable curves compatible with $\lambda$ have slope $b/a$ and spiral into the punctures of a twice-punctured disc.  
At most two such pairwise compatible allowable curves (obtained by applying $\kappa$ to a coinciding pair of tagged arcs) can be contained in each twice-punctured disc.
A coinciding pair of tagged arcs is determined by a choice of slope, an endpoint on which the taggings agree and a tagging at that endpoint.
Thus choosing a vertex of each twice-punctured disc and a tagging at each of these vertices completes the description of a maximal collection of pairwise compatible curves containing a closed curve.
\end{proof}

Maximal collections of pairwise compatible allowable curves containing no closed curves correspond to full-dimensional ($6$-dimensional) cones in the rational quasi-lamination fan~$\F_\mathbb{Q}(T_0)$.
By Theorem~\ref{Null Tangle consequences}, these are also the full-dimensional cones of the mutation fan $\F_B$.
Associating each cone with its corresponding tagged triangulation, two full-dimensional cones are adjacent in the mutation fan if and only if their tagged triangulations are related by flipping a single tagged arc.  
Flipping an arc means removing that arc and replacing it by the unique other arc that completes a tagged triangulation.
Details about flips of tagged arcs are found in \cite[Section 7]{fst}.  
These definitions and Table \ref{table:triangulations} imply that the adjacencies of $6$-dimensional cones in $\F_B$ are as follows: 
\begin{itemize}
\item Every type-I cone is adjacent to six type-II cones. 
\item Every type-II cone is adjacent to two type-I and four type-IV cones.
\item Every type-III cone is adjacent to two type-III and four type-IV cones.
\item Every type-IV cone is adjacent to two type-II, one type-III, two type-IV and one type-V cone.
\item Every type-V cone is adjacent to four type-IV and two type-VI cones.
\item Every type-VI cone is adjacent to six type-V cones.
\end{itemize}

Each maximal pairwise compatible set of allowable curves containing a closed curve also defines a maximal cone in the rational quasi-lamination fan.
These are the cones of type~VII in the sense of Tables~\ref{table:triangulations} and~\ref{table:curves}.
Theorem~\ref{Null Tangle consequences} says that these are cones in $\F_B$ as well. 
Since they are maximal in $\F_\rationals(T_0)$ and since the full-dimensional cones in $\F_\rationals(T_0)$ coincide with the full-dimensional cones in $\F_B$, we see that these cones are $5$-dimensional maximal cones of $\F_B$. 
Each of these type-VII cones is adjacent to four type-VII cones obtained by changing the spiral direction at both spiral points of a single non-closed curve.
The shear coordinate vector for each closed curve spans a ray that is contained in exactly sixteen type-VII cones.  
The link of the ray is combinatorially a $4$-dimensional crosspolytope.  

\section{Shear coordinates}\label{shear sec}  
In this section, we describe the computation of shear coordinates in the four-punctured sphere and explicitly derive the shear coordinate vectors listed in Theorem~\ref{unisphere thm}. 
Recall, however, that Theorem~\ref{unisphere thm} was stated for simplicity, not for efficiently listing the universal coefficients.
Here, we will list the shear coordinates efficiently.

To do so, we define some subgroups of $\br{(14),\, (25),\, (36),\,  (123)(456)}$. 
Let $Z=\br{(123)(456)}$, let $X=\br{(14)(25)(36)}$, and let $Y=\br{(14)(36), (25)(36)}$.
We take the convention of composing permutations from right to left, so that the set $Z \cdot Y$ consists of all permutations obtained by first applying a permutation in $Y$ and then applying a permutation in $Z$. 
Note however that this convention is irrelevant, as the generator of $X$ commutes with the generator of $Z$ and it can be checked that $Z \cdot Y = Y \cdot Z$ as sets.

\begin{theorem}\label{thm: shear}   
Let $T_0$ be the triangulation shown in Figure~\ref{tri and mat}.
The shear coordinates with respect to $T_0$ of allowable curves in the four-punctured sphere are listed irredundantly by taking each entry in the table below, letting $b/a$ vary over all standard forms of rational slopes in the range shown, constructing the vector shown, and applying each coordinate permutation in the set indicated.\\

\begin{center}
\begin{tabular}{llll}
$1.$\!&
$\begin{bmatrix} -\floor*{\frac{b-1}{2}}, \floor*{\frac{a}{2}}+1, \floor*{\frac{b-a}{2}}, -\floor*{\frac{b}{2}}, \floor*{\frac{a+1}{2}}, \floor*{\frac{b-a-1}{2}} \end{bmatrix}$
&&
$0<\frac{b}{a} \leq \infty$ \\
&&&$Z \cdot X$
\\[8pt]
$2.$\!&
$\begin{bmatrix}  -\floor*{\frac{b}{2}}-1, \floor*{\frac{a-1}{2}}, \floor*{\frac{b-a+1}{2}}, -\floor*{\frac{b+1}{2}}, \floor*{\frac{a}{2}}, \floor*{\frac{b-a}{2}}+1   \end{bmatrix}$
& &
$0\le\frac{b}{a} < \infty$ \\[2pt]
& &&
$Z \cdot X$ \\[8pt]
$3.$\!&
$\begin{bmatrix}  -\floor*{\frac{b}{2}}, \floor*{\frac{a+1}{2}}, \floor*{\frac{b-a+1}{2}}, -\floor*{\frac{b+1}{2}}, \floor*{\frac{a}{2}}, \floor*{\frac{b-a}{2}}   \end{bmatrix}$
&&
$0<\frac{b}{a} \leq \infty$\\[2pt]
&&&
$Z \cdot Y$\\[8pt]
$4.$\!&
$\begin{bmatrix}  -b, a, b-a, -b, a, b-a  \end{bmatrix}$
&&
$0<\frac{b}{a} \leq \infty$ \\[2pt]
&&&$Z$\end{tabular}
\end{center}
\end{theorem}

Before we prove Theorem~\ref{thm: shear}, we show that it has the correct relationship to Theorem~\ref{unisphere thm}.
\begin{proposition}\label{same list}
The set of vectors listed in Theorem~\ref{thm: shear} is the same as the set of vectors listed in Theorem~\ref{unisphere thm}.
\end{proposition}
\begin{proof}
We first notice that, setting aside for a moment permutations of the coordinates, the vectors listed in Item 2 of each theorem are the same: swapping $a$ and $b$ changes Item 2 of Theorem~\ref{unisphere thm} to Item 2 of Theorem~\ref{thm: shear}.

It is easily checked, keeping in mind the identity $\floor*{-\frac{k}{2}} = -\floor*{\frac{k+1}{2}}$ 
for $k \in \integers$, that each vector in Items 1 and 2 of the two theorems is fixed by exactly two of the permutations $(14)$, $(25)$, or $(36)$.
Thus, acting on each such vector by the group $X$ produces the same vectors as acting by $\br{(14),(25),(36)}$.
Since $Z \cdot X = X \cdot Z$, acting by $Z\cdot X$ produces the same vectors as acting by $\br{(14),(25),(36),(123)(456)}$.

Similarly, one can check that each vector in Item 3 of the two theorems is fixed by exactly one of the permutations $(14)$, $(25)$, or $(36)$.
Thus in this case, acting by $Y$ produces the same vectors as acting by $\br{(14),(25),(36)}$, and since $Z \cdot Y = Y \cdot Z$, 
acting by $Z\cdot Y$ produces the same vectors as acting by $\br{(14),(25),(36),(123)(456)}$.

Because the vectors in Item 4 are fixed by $\br{(14),(25),(36)}$, acting on these by $Z$ produces the same vectors as acting by $\br{(14),(25),(36),(123)(456)}$.
\end{proof}

The strategy for the proof of Theorem~\ref{thm: shear} is to calculate shear coordinates by lifting to the $\integers^2$-punctured plane and then reusing ideas from the analogous calculations for the once-punctured torus \cite[Proposition~5.1]{unitorus}. 
The triangulation $T_0=\set{\gamma_i}_{i=1}^6$ of $\Sp-4$ shown in Figure~\ref{tri and mat} lifts to a triangulation $\overline{T}_0$ of $\R^2-\Z^2$, shown in Figure~\ref{lift tri}. 
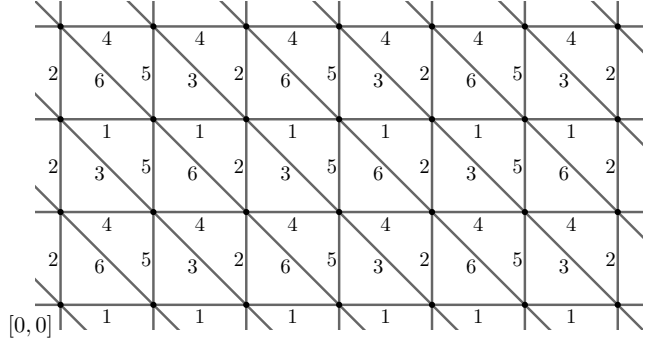
\begin{figure}\centering
\scalebox{0.7777}{\includegraphics{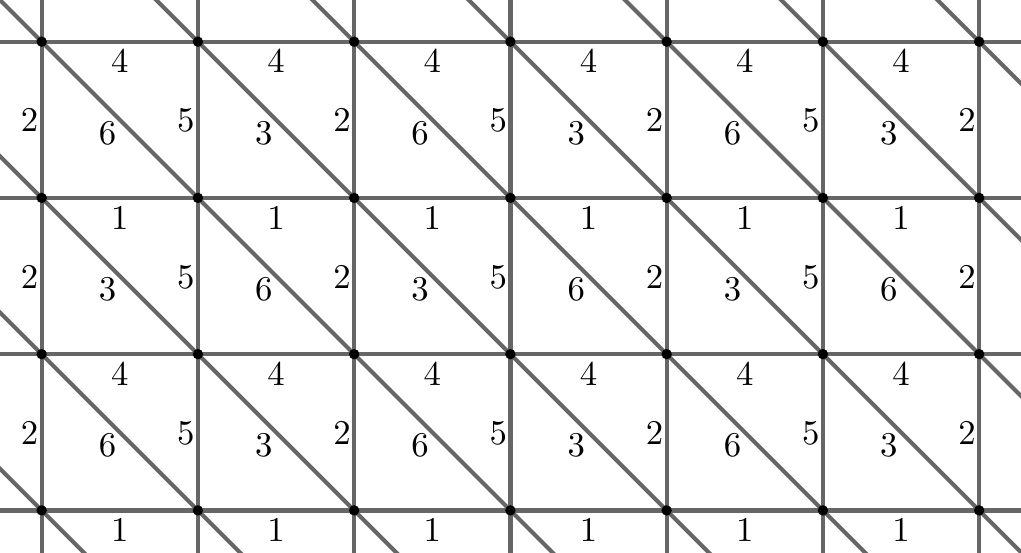}
\begin{picture}(0,0)(88,-7)
\put (-224,-8){\large$[0,0]$}
\end{picture}
}
\caption{The triangulation $\overline{T}_0$ obtained by lifting $T_0$ to the plane.}
\label{lift tri}
\end{figure}
By Proposition~\ref{prop: allowable curves}, a lift of an allowable curve $\lambda$ is either a line of rational slope in $\R^2-\Z^2$ or a line segment of rational slope with integer spiral points.
By Definition~\ref{shear def}, computing the shear coordinate $b_{\gamma_i}(T_0, \lambda)$ for a given arc $\gamma_i \in T_0$ amounts to counting and categorizing the intersections of a lift $\bar{\lambda}$ with lifts of $\gamma_i$ and summing the contributions of $-1, 0$, or $1$ for each crossing.

\begin{example}
\label{ex: shear1}
Consider the allowable curve $\lambda = \curve(2,3,\{v\notch_{00},v\notch_{01}\})$, one of whose lifts $\bar{\lambda}$ is portrayed in the left picture of Figure~\ref{shear ex fig 1}. 
Each arc in the figure is a lift $\bar{\gamma_i}$ of an arc $\gamma_i \in T_0$. Lifts $\bar{\gamma_i}$ which intersect $\bar{\lambda}$ in one of the two nontrivial ways depicted in Figure~\ref{shear fig} are labeled in black with the numbering from Figure~\ref{tri and mat}.
The corresponding nonzero contributions to shear coordinates are shown in red, adding up to $\b(T_0, \lambda)=[-1,2,0,-1,1,0]$. 
\end{example}

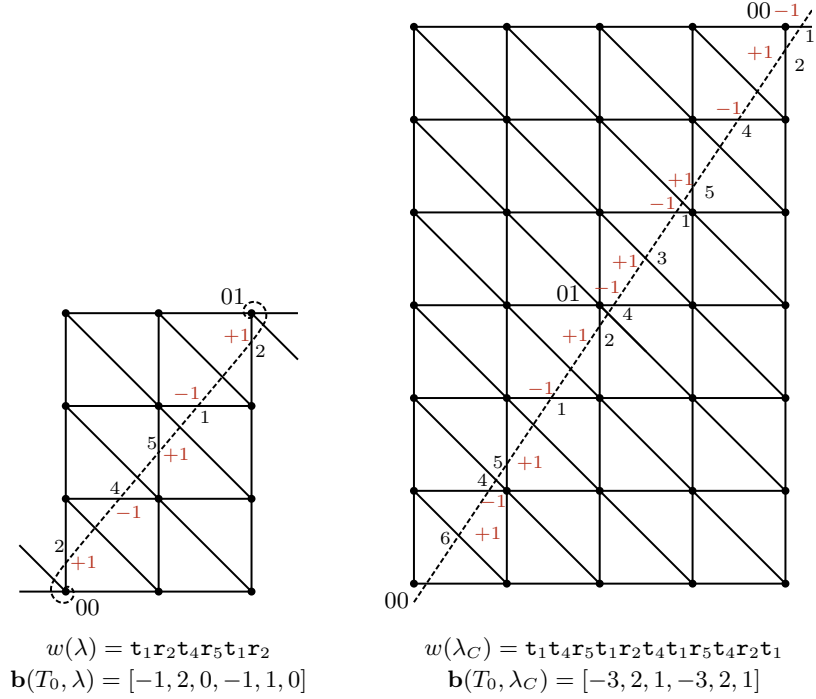
\begin{figure}\centering
\begin{tabular}{cccc}
\scalebox{1}{
\includegraphics{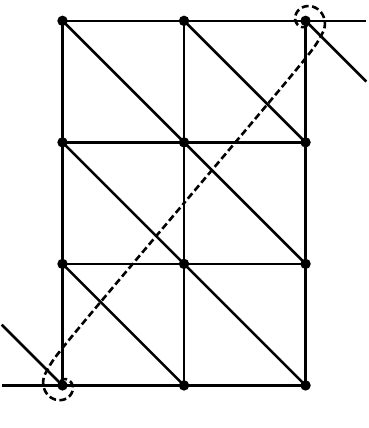}
\begin{picture}(0,0)(88,-7)
\put(0,-4){\small$00$}
\put(55,114){\small$01$}
\put(-8,20){\scriptsize$2$}
\put(-2,14){\textcolor{BrickRed}{\scriptsize$+1$}}
\put(13,42){\scriptsize$4$}
\put(15,33){\textcolor{BrickRed}{\scriptsize$-1$}}
\put(27,59){\scriptsize$5$}
\put(33,55){\textcolor{BrickRed}{\scriptsize$+1$}}
\put(47,69){\scriptsize$1$}
\put(37,78){\textcolor{BrickRed}{\scriptsize$-1$}}
\put(56,100){\textcolor{BrickRed}{\scriptsize$+1$}}
\put(68,94){\scriptsize$2$}
\end{picture}}
&&&
{\includegraphics{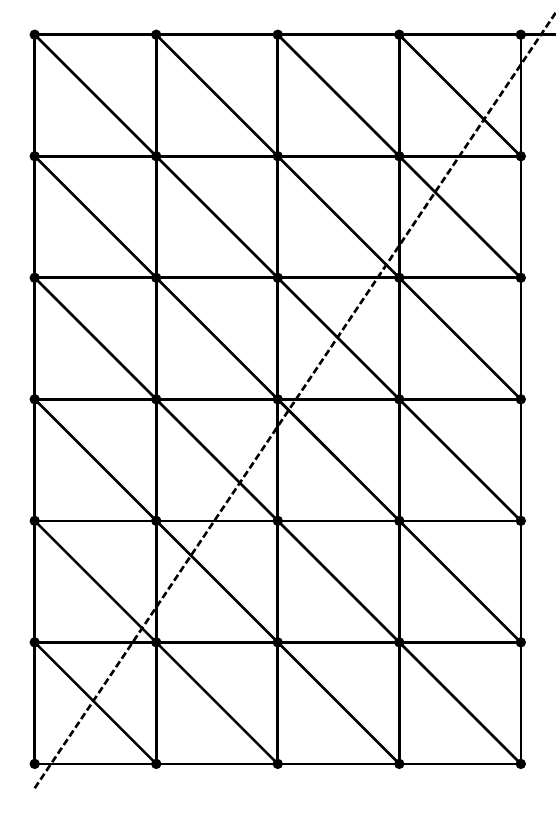} 
\begin{picture}(0,238)(88,-7)
\put(-77,-1){\small$00$}
\put(-12,115){\small$01$}
\put(60,222){\small$00$}
\put(-56,23){\scriptsize$6$}
\put(-43,25){\textcolor{BrickRed}{\scriptsize$+1$}}
\put(-42,45){\scriptsize$4$}
\put(-41,37){\textcolor{BrickRed}{\scriptsize$-1$}}
\put(-36,52){\scriptsize$5$}
\put(-27,52){\textcolor{BrickRed}{\scriptsize$+1$}}
\put(-12,72){\scriptsize$1$}
\put(-23,80){\textcolor{BrickRed}{\scriptsize$-1$}}
\put(6,98){\scriptsize$2$}
\put(-10,100){\textcolor{BrickRed}{\scriptsize$+1$}}
\put(13,108){\scriptsize$4$}
\put(2,118){\textcolor{BrickRed}{\scriptsize$-1$}}
\put(26,129){\scriptsize$3$}
\put(9,128){\textcolor{BrickRed}{\scriptsize$+1$}}
\put(35,143){\scriptsize$1$}
\put(23,150){\textcolor{BrickRed}{\scriptsize$-1$}}
\put(44,154){\scriptsize$5$}
\put(30,159){\textcolor{BrickRed}{\scriptsize$+1$}}
\put(58,177){\scriptsize$4$}
\put(48,186){\textcolor{BrickRed}{\scriptsize$-1$}}
\put(78,202){\scriptsize$2$}
\put(60,207){\textcolor{BrickRed}{\scriptsize$+1$}}
\put(82,213){\scriptsize$1$}
\put(70,223){\textcolor{BrickRed}{\scriptsize$-1$}}
\end{picture}}\\
{\small$\w(\lambda) = \t_1\r_2\t_4\r_5\t_1\r_2$}
&&&
{\small$\w(\lambda_C)= \t_1\t_4\r_5\t_1\r_2\t_4\t_1\r_5\t_4\r_2\t_1$}\\
{\small$\b(T_0, \lambda)=[-1,2,0,-1,1,0]$}
&&&
{\small$\b(T_0, \lambda_C)=[-3,2,1,-3,2,1]$}
\end{tabular}
\caption{Some shear coordinate calculations.}
\label{shear ex fig 1}
\end{figure}

We have chosen the triangulation $T_0$ so that it lifts to the same triangulation of $\R^2-\Z^2$ that was considered in \cite{unitorus}. 
The latter was a lift of a triangulation of $\T-1$ that we will call $T'_0$, pictured in Figure~\ref{tri and mat torus}.
\begin{figure}\centering
\scalebox{1}{
\raisebox{-5pt}{\includegraphics{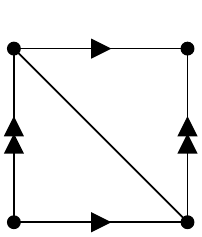}
\begin{picture}(0,0)(56,0)
\put(20,-6){\small$\bar\gamma_1$}
\put(-14,27){\small$\bar\gamma_2$}
\put(17,23){\small$\bar\gamma_3$}
\put(20,60){\small$\bar\gamma_1$}
\put(52,27){\small$\bar\gamma_2$}
\end{picture}}}
\qquad\quad
\raisebox{21pt}{$
\begin{bmatrix*}[r]
0 & 2 & -2 \\
-2 & 0 & 2 \\
2 & -2 & 0
\end{bmatrix*}
$}
\caption{A triangulation $T'_0$ of the once-punctured torus and its signed adjacency matrix}
\label{tri and mat torus}
\end{figure}
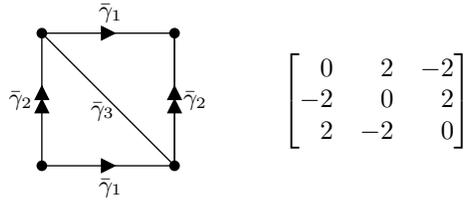
(It is important to keep in mind that in~\cite{unitorus}, $\T-1$ is realized as the quotient of $\reals^2 - \integers^2$ modulo the group of all integer translations.
By contrast, in Section~\ref{covering sec}, we realized $\T-1$ the quotient of $\R^2-(2\Z)^2$ modulo the group of all even-integer translations.
The lift of $T_0'$ is with respect to the covering map of $\T-1$ by $\reals^2-\integers^2$.)
Each `horizontal' arc labeled by either $1$ or $4$ in $\overline{T}_0$ is a lift of one arc in $T'_0$.
Similarly each `vertical' arc labeled by either $2$ or $5$ is a lift of another arc in $T'_0$, and each `diagonal' arc labeled $3$ or~$6$ is a lift of the third arc in $T'_0$.

To calculate shear coordinates, we first restrict our attention to curves that have finite positive slopes and are either closed or spiral counterclockwise  into $v_{00}$.
Each such curve $\lambda$ is either $\curve(a,b)$ with $0<\frac{b}{a}<\infty$ or $\curve(a,b,E)$ with $v\notch_{00} \in E$ and $0<\frac{b}{a}<\infty$ and hence has a lift $\bar{\lambda}$ in $\R^2-\Z^2$ which passes through the square with vertices $[0,0]$ and $[1,1]$ and either intersects the segment between $[0,0]$ and $[1,0]$ or has a spiral point at the origin.
We describe the interaction between this lift $\bar{\lambda}$ and the lifted triangulation $\overline{T}_0$ with a word $\w(\lambda)$ in letters $\mathtt{r}_2$, $\mathtt{r}_5$, $\mathtt{t}_1$, and $\mathtt{t}_4$.
It will be convenient to think of the subscripts on $\r$ and $\t$ as ``decorations,'' so that, for example, we can count occurrences of $\r$ in a word, ignoring subscripts, but also more specifically count instances of $\r_2$ in a word.  

To define $\w(\lambda)$, we first define a subword $\w'(a,b)$ inherent to the slope $b/a$ and independent of whether $\lambda$ is closed or spirals into $v_{00}$.
Since $b/a$ is positive and finite and $\gcd(a,b)=1$, the line segment connecting $[0,0]$ to $[a,b]$ starts in the unit square with vertices $[0,0]$ and $[1,1]$ and then exits that square through the right edge or the top edge (unless $a=b=1$).
It then exits another square through the right or top, and so forth until it reaches $[a,b]$.
The word $\w'(a,b)$ records $\r$ when the segment exits a square through the right side and $\t$ when it exits through the top, and therefore consists of $a-1$ instances of $\r$ and $b-1$ instances of~$\t$. 
We set $\w'(1,1)$ to be the empty word.
Subscripts keep track of the labels on these right and top edges in Figure~\ref{lift tri}.
For example, $\w'(2,3)=\t_4\r_5\t_1$.  

Now we define $\w(\curve(a,b,E))$ by appending prefixes and suffixes to $\w'(a,b)$ according to the tagging of $v_{ab}$ in $E$: 
\begin{align*}
\w(\curve(a,b,\{v\notch_{00}, v\notch_{ab}\}))&=\t_1\r_2\w'(a,b)\r_k\\
\w(\curve(a,b,\{v\notch_{00}, v_{ab}\}))&=\t_1\r_2\w'(a,b)\t_{\ell}
\end{align*}
The subscript $k$ is $2$ if $a$ is even and $5$ if $a$ is odd, and the subscript $\ell$ is $1$ if $b$ is even and $4$ if $b$ is odd.

\begin{example}
\label{ex: shear2}  
Consider $\lambda = \curve(2,3,\{v\notch_{00},v\notch_{01}\})$ as in Example~\ref{ex: shear1}. 
Since $\w'(a,b) =\t_4\r_5\t_1$, we have $\w(\lambda) =\textcolor{BrickRed}{\t_1\r_2}\t_4\r_5\t_1\textcolor{BrickRed}{\r_2}$. The prefix and suffix are colored red for clarity. 
We decorate the suffix as ${\r_2}$ since $a=2$ is even, because the righthand edge of the unit square with vertices $[a-1, b-1]=[1,2]$ and $[a,b]=[2,3]$ is a lift of arc $\gamma_2$. 
\end{example}

The word $\w(\curve(a,b))$ is defined by taking a lift $\overline{\curve(a,b)}$ that intersects the segment from $[0,0]$ to $[1,0]$, recording a $\t$ for exiting the square with corners $[0,0]$ and $[1,-1]$, and then recording $\r$ and $\t$ until we reach an even-integer translate of the starting point, where we again record a $\t$.
\[\w(\curve(a,b))=\t_1\w'(a,b)\r_k\t_{\ell}\tilde{\w}'(a,b)\r_2\t_1\]
The subscripts $k$ and $\ell$ depend on the parity of $a$ and $b$ as described above, and $\tilde{\w}'(a,b)$ means $\w'(a,b)$ read backwards.

\begin{example}
\label{ex: shear3}
The right picture in Figure~\ref{shear ex fig 1} depicts a lift of $\lambda_C = \curve(2,3)$ that intersects the line segment connecting $[0,0]$ and $[1,0]$. 
We have $\w(\lambda_C)=\textcolor{BrickRed}{\t_1}\t_4\r_5\t_1\textcolor{BrickRed}{\r_2\t_4}\t_1\r_5\t_4\textcolor{BrickRed}{\r_2\t_1}$. 
The prefix, infix, and suffix are colored red for clarity. 
\end{example}

\begin{figure}
\begin{tabular}{c|c|c|c}
\scalebox{0.8}{\includegraphics{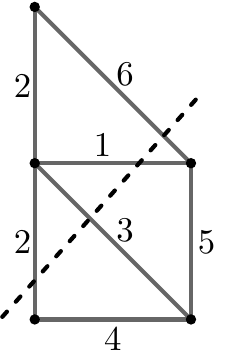}}&\scalebox{0.8}{\includegraphics{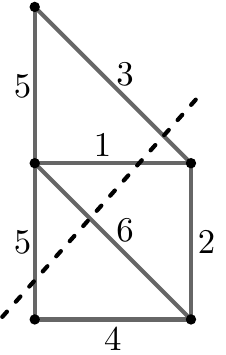}}&\scalebox{0.8}{\includegraphics{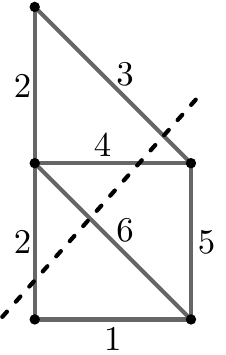}}&\scalebox{0.8}{\includegraphics{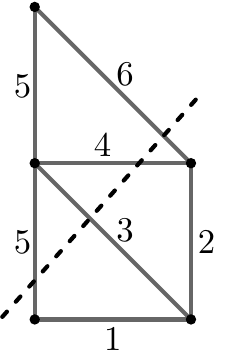}}\\
${\tt r_2t_1}$&${\tt r_5t_1}$&${\tt r_2t_4}$&${\tt r_5t_4}$\\
$[-1,0,0,0,0,0]$&$[-1,0,0,0,0,0]$&$[0,0,0,-1,0,0]$&$[0,0,0,-1,0,0]$\\[3pt]\hline&&&\\[-6pt]
\scalebox{0.8}{\includegraphics{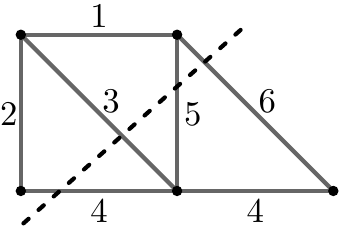}}&\scalebox{0.8}{\includegraphics{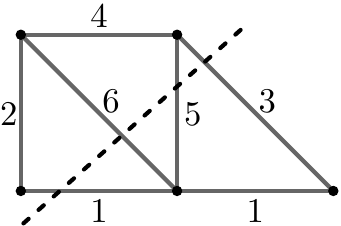}}&\scalebox{0.8}{\includegraphics{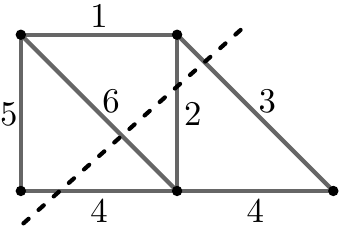}}&\scalebox{0.8}{\includegraphics{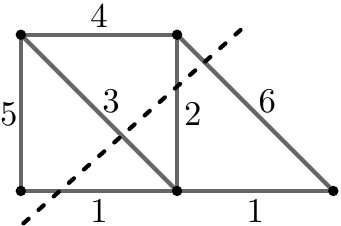}} \\
${\tt t_4r_5}$&${\tt t_1r_5}$&${\tt t_4r_2}$&${\tt t_1r_2}$ \\
$[0,0,0,0,1,0]$&$[0,0,0,0,1,0]$&$[0,1,0,0,0,0]$&$[0,1,0,0,0,0]$\\[3pt]\hline&&&\\[-6pt]
\scalebox{0.8}{\includegraphics{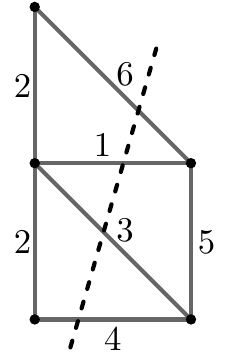}}&\scalebox{0.8}{\includegraphics{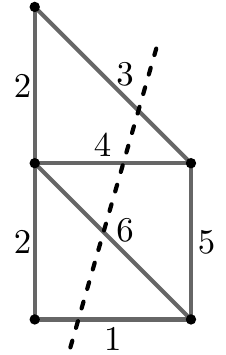}}&\scalebox{0.8}{\includegraphics{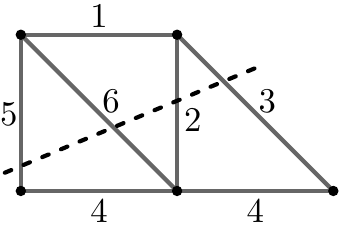}}&\scalebox{0.8}{\includegraphics{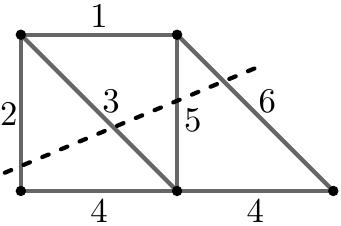}}\\
${\tt t_4t_1}$&${\tt t_1t_4}$&${\tt r_5r_2}$&${\tt r_2r_5}$\\
$[-1,0,1,0,0,0]$&$[0,0,0,-1,0,1]$&$[0,1,0,0,0,-1]$&$[0,0,-1,0,1,0]$\\[4pt]
\framebox{or}&\framebox{or}&\framebox{or}&\framebox{or}\\[4pt]
\scalebox{0.8}{\includegraphics{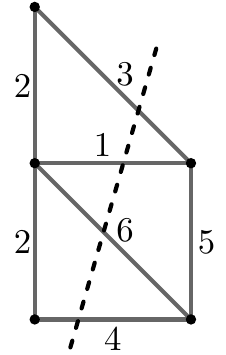}}&\scalebox{0.8}{\includegraphics{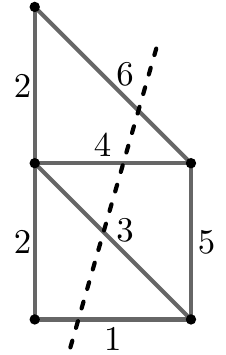}}&\scalebox{0.8}{\includegraphics{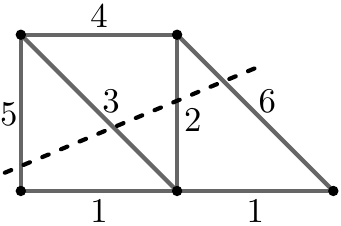}}&\scalebox{0.8}{\includegraphics{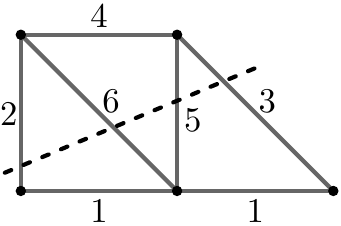}}\\
${\tt t_4t_1}$&${\tt t_1t_4}$&${\tt r_5r_2}$&${\tt r_2r_5}$\\
$[-1,0,0,0,0,1]$&$[0,0,1,-1,0,0]$&$[0,1,-1,0,0,0]$&$[0,0,0,0,1,-1]$ \\[-7pt]&&&
\end{tabular}
\caption{Shear coordinate contributions from consecutive pairs}
\label{fig: shear}
\end{figure}

Four of the six entries of the shear coordinates of $\lambda$ can be read off from $\w(\lambda)$ by recording and summing shear coordinate contributions for each consecutive pair of letters. 
These consecutive pairs encode interactions of the lift $\bar{\lambda}$ with each arc $\bar{\gamma_i} \in \overline{T}_0$, and therefore of $\lambda$ with each arc $\gamma_i \in T_0$.
Using Definition~\ref{shear def}, we record contributions to the shear coordinates as illustrated in Figure~\ref{fig: shear}.
By inspection of Figure~\ref{fig: shear}, we then obtain the following lemma. 

\begin{lemma}\label{lem: horiz. and vert. shear}
Suppose $b/a$ is the standard form of a positive finite slope and $\lambda$ is either $\curve(a,b)$ or $\curve(a,b,E)$ with $v\notch_{00}\in E$, and write $\w$ for $\w(\lambda)$.
\begin{align*}
b_{\gamma_1}(T_0, \lambda) & = -1 \times \left(\#\text{consecutive pairs }\r\t_1\text{ in }\w+ \#\text{consecutive pairs }\t\t_1\text{ in }\w\right) \\
& = -(\#\text{ non-leading instances of }\t_1\text{ in }\w) \\
b_{\gamma_2}(T_0, \lambda) & = 1 \times \left(\#\text{consecutive pairs }\t\r_2\text{ in }\w + \# \text{consecutive pairs }\r\r_2\text{ in }\w \right)\\
& = \# \text{ non-leading instances of }\r_2\text{ in }\w \\
b_{\gamma_4}(T_0, \lambda) & = -1 \times \left(\#\text{consecutive pairs }\r\t_4\text{ in }\w+ \#\text{consecutive pairs }\t\t_4\text{ in }\w\right) \\
& = -(\# \text{ non-leading instances of }\t_4\text{ in }\w) \\
b_{\gamma_5}(T_0, \lambda) & = 1 \times \left(\#\text{consecutive pairs }\t\r_5\text{ in }\w + \# \text{consecutive pairs }\r\r_5\text{ in }\w \right)\\
& = \# \text{ non-leading instances of }\r_5\text{ in }\w
\end{align*}
\end{lemma}

Unfortunately, the third and sixth coordinates of $\b(T_0, \lambda)$ are not read off as directly from $\w(\lambda)$.  
The problem is that, as recorded at the bottom of Figure~\ref{fig: shear}, the information in $\w(\lambda)$ is not enough to determine which diagonal arc in the plane is crossed at a certain point in the word.
We handle this subtlety below in Lemma~\ref{lem: shear diagonals}.

\begin{lemma}
\label{lem: shear diagonals}
Suppose $b/a$ is the standard form of a positive finite slope and $\lambda$ is either $\curve(a,b)$ or $\curve(a,b,E)$ with $v\notch_{00}\in E$. 
\begin{align*}
\bigl|b_{\gamma_3}(T_0, \lambda) + b_{\gamma_6}(T_0, \lambda)\bigr|  &= \# \text{consecutive pairs of doubles }\r\r\text{ and }\t\t\text{ in }\w(\lambda)\\
\bigl|b_{\gamma_3}(T_0, \lambda) - b_{\gamma_6}(T_0, \lambda)\bigr|& \leq 1
\end{align*}
\end{lemma}

\begin{proof}
Observe from Figure \ref{fig: shear} that non-trivial contributions to $b_{\gamma_3}(T_0, \lambda)$ or to $b_{\gamma_6}(T_0, \lambda)$ occur precisely at consecutive pairs $\t\t$ and $\r\r$ in $\w(\lambda)$. 
If $b/a>1$ then the word $\w'(a,b)$ begins and ends with $\t$ and may contain consecutive pairs $\t\t$, but may not contain consecutive pairs $\r\r$.
If $b/a<1$ then $\w'(a,b)$ begins and ends with $\r$ and may contain $\r\r$, but may not contain $\t\t$.
When $\w(\lambda)$ is formed as described above it retains the property that pairs $\r\r$ and $\t\t$ cannot both appear.
The first equation follows by inspection of Figure \ref{fig: shear} as all non-trivial contributions to $b_{\gamma_3}(T_0, \lambda)$ or $b_{\gamma_6}(T_0, \lambda)$ are of the same sign.

The intersections of $\bar{\lambda}$ with diagonal arcs in $\overline{T}_0$ alternate between lifts of $\gamma_3$ and lifts of $\gamma_6$. To prove that $b_{\gamma_3}(T_0, \lambda)$ or $b_{\gamma_6}(T_0, \lambda)$ differ by at most $1$, we check that the subset of {\em non-trivial} intersections alternates as well. 
To this end, we show that all trivial diagonal intersections, with the possible exception of those given by endpoint spirals and therefore not occurring between non-trivial intersections, occur in twos. 
Since trivial diagonal intersections correspond to consecutive pairs $\t\r$ and $\r\t$, it suffices to show that interior occurrences of these pairs in $w$ come in sets of two. 

If $0<b/a<1$, then $\w(\lambda)$ contains only consecutive pairs from the set $\set{\r\t,\t\r,\r\r}$ and $\w'(a,b)$ begins and ends with $\r$. 
It follows that each instance of $\t$ in the interior of $\w(\lambda)$, and hence each interior consecutive pair $\r\t$, occurs in a subword $\r\t\r$ corresponding to two trivial diagonal intersections. 
Analogously, if $1<b/a<\infty$, then $\w'(a,b)$ begins and ends with {\tt t} and contains only consecutive pairs from the set $\{\r\t, \t\r, \t\t\}$, implying that each interior instance of {\tt r} occurs in a subword {\tt trt}.  
We conclude that either $b_{\gamma_3}(T_0, \lambda) = b_{\gamma_6}(T_0, \lambda)$ or $\big|b_{\gamma_3}(T_0, \lambda) - b_{\gamma_6}(T_0, \lambda)\big|=1$.
\end{proof}

With Lemmas~\ref{lem: horiz. and vert. shear} and \ref{lem: shear diagonals} in hand, we have reduced the computation of $\b(T_0, \lambda)$ in 
the special cases of a finite, positively-sloped curve $\lambda$ which is closed or spirals counterclockwise into $v_{00}$
 to the problem of counting instances of non-leading decorated letters and decorated consecutive letter pairs in the word $\w(\lambda)$.   
We will solve the counting problem and then use symmetry to extend our results to curves spiraling clockwise into $v_{00}$ to prove the following proposition.

\begin{proposition}
\label{prop: positive shear} 
The following curves have the following shear coordinates with respect to $T_0$:

\noindent
\begin{tabular}{lll}
$1.$ 
& \!\!\!\small$\begin{array}{l}\curve(a,b,\{v\notch_{00}, v\notch_{ab}\})\\0<b/a\le\infty\end{array}$ 
& $\begin{bmatrix} -\floor*{\frac{b-1}{2}}, \floor*{\frac{a}{2}}+1, \floor*{\frac{b-a}{2}}, -\floor*{\frac{b}{2}}, \floor*{\frac{a+1}{2}}, \floor*{\frac{b-a-1}{2}} \end{bmatrix}$ 
\\[8pt]
$2.$ 
& \!\!\!\small$\begin{array}{l}\curve(a,b,\{v\notch_{00}, v_{ab}\})\\0\le b/a\le\infty\end{array}$ 
& $\begin{bmatrix}  -\floor*{\frac{b}{2}}, \floor*{\frac{a+1}{2}}, \floor*{\frac{b-a+1}{2}}, -\floor*{\frac{b+1}{2}}, \floor*{\frac{a}{2}}, \floor*{\frac{b-a}{2}}   \end{bmatrix}$ 
\\[8pt]
$3.$ 
& \!\!\!\small$\begin{array}{l}\curve(a,b,\{v_{00}, v\notch_{ab}\})\\0\le b/a\le\infty\end{array}$ 
& $\begin{bmatrix}  -\floor*{\frac{b+1}{2}}, \floor*{\frac{a}{2}}, \floor*{\frac{b-a}{2}}, -\floor*{\frac{b}{2}}, \floor*{\frac{a+1}{2}}, \floor*{\frac{b-a+1}{2}}   \end{bmatrix}$ 
\\[8pt]
$4.$ 
& \!\!\!\small$\begin{array}{l}\curve(a,b,\{v_{00}, v_{ab}\})\\0\le b/a<\infty\end{array}$ 
& $\begin{bmatrix}  -\floor*{\frac{b}{2}}-1, \floor*{\frac{a-1}{2}}, \floor*{\frac{b-a+1}{2}}, -\floor*{\frac{b+1}{2}}, \floor*{\frac{a}{2}}, \floor*{\frac{b-a}{2}}+1   \end{bmatrix}$ 
\\[8pt]
$5.$ 
&  \!\!\!\small$\begin{array}{l}\curve(a,b)\\0\le b/a\le\infty\end{array}$ 
& $\begin{bmatrix}  -b, a, b-a, -b, a, b-a  \end{bmatrix}$ 
\end{tabular}
\end{proposition}

\begin{proof}
One can check directly that the given formulas work, or don't work, for the cases $b/a=0$ and $b/a=\infty$ as indicated. 
For the other cases of positive finite slope and $v\notch_{00}\in E$, we argue using the word $\w(a,b)$.
If $\lambda=\curve(a,b,\{v\notch_{00}, v\notch_{ab}\})$, then $\w(\lambda)=\t_1\r_2\w'(a,b)\r_k$, with $k=2$ if $a$ is even or $k=5$ if $a$ is odd.
Since $\w'(a,b)$ has $b-1$ instances of $\t$ and $a-1$ instances of $\r$, we see that $\w(\lambda)$ has $b-1$ non-leading instances of $\t$, alternating $\t_4, \t_1, \t_4, \ldots$, and $a+1$ non-leading instances of $\r$, alternating $\r_2, \r_5, \r_2, \ldots$.
Lemma~\ref{lem: horiz. and vert. shear} thus gives us the first, second, fourth, and fifth coordinates of $\b(T_0, \lambda)$:

\begin{tabular}{cl}
Horizontal Arcs & $b_{\gamma_1}(T_0, \lambda)= -1 \times \floor*{\frac{b-1}{2}} = - \floor*{\frac{b-1}{2}}$\\[8pt]
& $b_{\gamma_4}(T_0, \lambda)= -1 \times \ceil*{\frac{b-1}{2}} = - \ceil*{\frac{b-1}{2}} = -\floor*{\frac{b}{2}}$\\[12pt]
Vertical Arcs & $b_{\gamma_2}(T_0, \lambda)= 1 \times \ceil*{\frac{a+1}{2}} = \floor*{\frac{a+2}{2}} = \floor*{\frac{a}{2}}+1$\\[8pt]
& $b_{\gamma_5}(T_0, \lambda)= 1 \times \floor*{\frac{a+1}{2}} = \floor*{\frac{a+1}{2}}$
\end{tabular}\\

To obtain the third and sixth coordinates of $\b(T_0, \lambda)$, we count the number of consecutive pairs {\tt rr} (if $0<b/a \leq 1$), or {\tt tt} (if $b/a > 1$), and then apply Lemma~\ref{lem: shear diagonals}. 
Since the word $\w(\lambda)$ is of length $b+a+1$, it contains $b+a$ consecutive letter pairs. 
If $0<b/a\leq1$, then since $\w(\lambda)$ contains no consecutive pairs $\t\t$ and $b-1$ non-leading instances of $\t$, it contains $b$ consecutive letter pairs $\t\r$, $b-1$ consecutive letter pairs $\r\t$, and hence $b+a-b-(b-1)=a-b+1$ consecutive letter pairs $\r\r$.
Each of these corresponds to a non-trivial diagonal crossing of type $-1$, alternating $\gamma_6, \gamma_3, \ldots,$ for a total of $-1(a-b+1)=b-a-1< 0$. 
If $b/a >1$, then $\w(\lambda)$ contains $a+1$ consecutive letter pairs $\t\r$, $a$ consecutive letter pairs $\r\t$, and hence $b+a-(a+1)-a=b-a-1$ consecutive letter pairs $\t\t$, each corresponding to a non-trivial diagonal crossing of type $1$, now alternating $\gamma_3, 
\gamma_6, \ldots$, again for a total of $b-a-1\geq 0$. 
Regardless, we obtain\\

\begin{tabular}{cl}
Diagonal Arcs & $b_{\gamma_3}(T_0, \lambda) =\ceil*{\frac{b-a-1}{2}} = \floor*{\frac{b-a}{2}}$\\[8pt]
& $b_{\gamma_6}(T_0, \lambda) =\floor*{\frac{b-a-1}{2}}$
\end{tabular}\\

To get $\w(\curve(a,b,\{v\notch_{00}, v_{ab}\}))$ from $\w(\curve(a,b,\{v\notch_{00}, v\notch_{ab}\}))$, we delete the final $r_k$ (with $k=2$ if $a$ is even and $k=5$ if $a$ is odd) and append a final $t_\ell$ (with $\ell=1$ if $b$ is even and $\ell=4$ if $b$ is odd). 
Using Lemma~\ref{lem: horiz. and vert. shear}, we see that this change in words 
subtracts $1$ from the $\gamma_1$-coordinate if $b$ is even,
subtracts $1$ from the $\gamma_2$-coordinate if $a$ is even,
subtracts $1$ from the $\gamma_4$-coordinate if $b$ is odd, and
subtracts $1$ from the $\gamma_5$-coordinate if $a$ is odd.
This change in words adds a $\t\t$ pair if $b>a$, thus adding a nontrivial intersection with $\gamma_3$ if $b-a$ is odd or with $\gamma_6$ if $b-a$ is even.
The change destroys an $\r\r$ pair if $b\leq a$, thus destroying a nontrivial intersection with $\gamma_3$ if $b-a$ is odd or with $\gamma_6$ if $b-a$ is even.
We see that the formula for shear coordinates of $\w(\curve(a,b,\{v\notch_{00}, v_{ab}\}))$ follows from the formula for shear coordinates of $\w(\curve(a,b,\{v\notch_{00}, v\notch_{ab}\}))$.

Reflecting the plane about the line $y=x$ maps lifts of $\gamma_1$ to lifts of $\gamma_2$, maps lifts of $\gamma_4$ to lifts of $\gamma_5$, fixes the set of lifts of $\gamma_3$ and fixes the set of lifts of $\gamma_6$.
The reflection also takes a lifted curve of slope $b/a$ to a lifted curve of slope $a/b$ with reversed spiral directions, and changes intersections of type $1$ of the lifted curve with arcs in $\overline{T}_0$ 
(see Figure~\ref{shear fig}) to intersections of type $-1$ and vice versa. 
Thus the shear coordinate vector for $\curve(a,b,\{v_{00},v_{ab}\})$ is obtained from the shear coordinate vector for $\curve(b,a,\{v\notch_{00},v\notch_{ba}\})$ by negating all coordinates and permuting the coordinates via $(12)(45)$.
(See Example~\ref{ex: shear4}.)
The shear coordinate vector for $\curve(a,b,\{v\notch_{00},v_{ab}\})$ is obtained from that for $\curve(b,a,\{v_{00},v\notch_{ba}\})$ in the same manner. 
Making use of the identity $\floor*{-\frac{k}{2}} = -\floor*{\frac{k+1}{2}}$ for $k \in \integers$, we deduce the third and fourth items in the proposition from the first two.
The calculation of shear coordinates of $\curve(a,b)$ is very similar to that for $\curve(a,b,\{v\notch_{00}, v\notch_{ab}\})$, and we omit the details.
\end{proof}

\begin{figure}\centering
\scalebox{1}{
\begin{tabular}[t]{ccccc}
\includegraphics{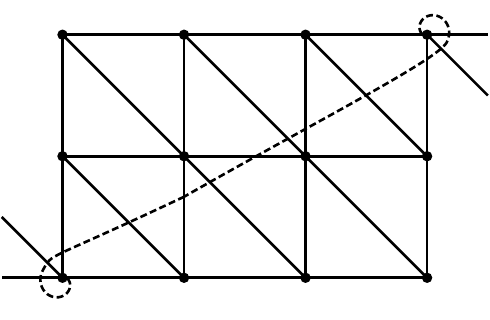}
\begin{picture}(0,0)(88, -7)
\put(-55,-4){\small$00$}
\put(55,81){\small$10$}
\put(-37,8){\scriptsize$1$}
\put(-35,-4){\textcolor{BrickRed}{\scriptsize$-1$}}
\put(-8,30){\scriptsize$5$}
\put(-3,24){\textcolor{BrickRed}{\scriptsize$+1$}}
\put(13,42){\scriptsize$4$}
\put(15,34){\textcolor{BrickRed}{\scriptsize$-1$}}
\put(27,50){\scriptsize$2$}
\put(33,45){\textcolor{BrickRed}{\scriptsize$+1$}}
\put(52,69){\scriptsize$1$}
\put(40,78){\textcolor{BrickRed}{\scriptsize$-1$}}
\end{picture}
&&&
\includegraphics{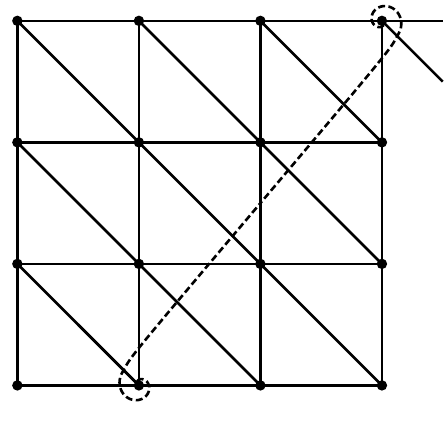}
\begin{picture}(0,128)(88, -7)
\put(-38,-4){\small$00$}
\put(0,-4){\small$10$}
\put(55,114){\small$11$}
\put(-8,20){\scriptsize$5$}
\put(-2,14){\textcolor{BrickRed}{\scriptsize$+1$}}
\put(13,42){\scriptsize$4$}
\put(15,33){\textcolor{BrickRed}{\scriptsize$-1$}}
\put(27,59){\scriptsize$2$}
\put(33,55){\textcolor{BrickRed}{\scriptsize$+1$}}
\put(47,69){\scriptsize$1$}
\put(37,78){\textcolor{BrickRed}{\scriptsize$-1$}}
\put(56,100){\textcolor{BrickRed}{\scriptsize$+1$}}
\put(68,94){\scriptsize$5$}
\end{picture}\\
{\small$\b(T_0, \lambda')=[-2,1,0,-1,1,0]$}
&&&
{\small$\b(T_0, \lambda''')=[-1,0,0,-1,1,2]$}
\end{tabular}}\\[10pt]
\includegraphics{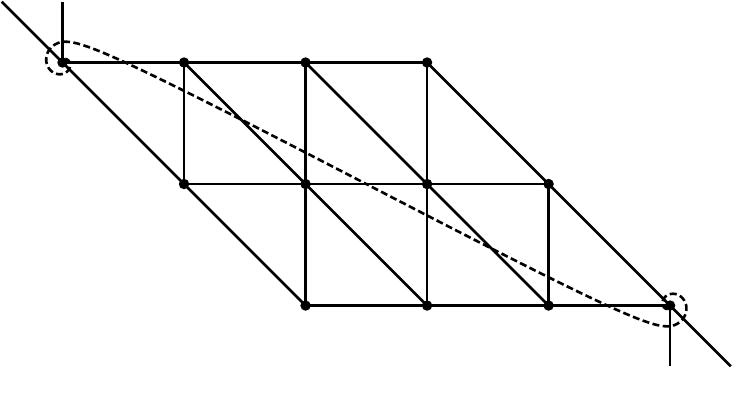}
\begin{picture}(0,0)(193, -100)
\put(-8,-11){\small$00$}
\put(169,-64){\small$10$}
\put(18,2){\scriptsize$1$}
\put(6,-7){\textcolor{BrickRed}{\scriptsize$+1$}}
\put(47,-14){\scriptsize$6$}
\put(39,-24){\textcolor{BrickRed}{\scriptsize$-1$}}
\put(85,-33){\scriptsize$4$}
\put(77,-42){\textcolor{BrickRed}{\scriptsize$+1$}}
\put(120,-52){\scriptsize$3$}
\put(110,-59){\textcolor{BrickRed}{\scriptsize$-1$}}
\put(153,-67){\textcolor{BrickRed}{\scriptsize$+1$}}
\put(151,-76){\scriptsize$1$}
\put(35, -92){\small$\b(T_0, \lambda'')=[2,0,-1,1,0,-1]$}
\end{picture}
\caption{More shear coordinate calculations.}
\label{shear ex fig 2}
\end{figure}
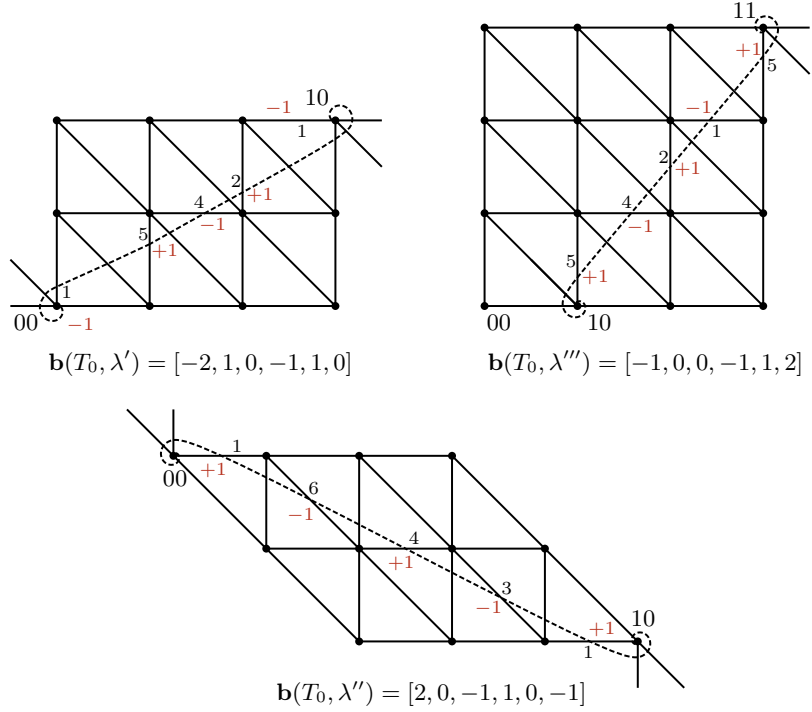

\begin{example}
\label{ex: shear4} 
Consider the allowable curve $\lambda'=\curve(3,2,\{v_{00},v_{10}\})$, with lift $\bar{\lambda}'$ depicted in the top-left picture of Figure~\ref{shear ex fig 2}.
We compute the shear coordinate vector $\b(T_0, \lambda')=[-2,1,0,-1,1,0]$. 
The curve $\lambda'$ is obtained from the allowable curve $\lambda$ of Examples~\ref{ex: shear1} and \ref{ex: shear2} via reflection about the line $y=x$, corresponding to the relation of $\b(T_0, \lambda')$ to $\b(T_0, \lambda)$ via coordinate negation and permutation by $(12)(45)$.
\end{example}

\begin{proof}[Proof of Theorem~\ref{thm: shear}]
\label{proof of shear}
Theorem~\ref{q-lam bij} implies that the map from allowable curves to shear coordinates is one-to-one. 
By Proposition~\ref{prop: allowable curves}, allowable curves in the four-punctured sphere are in bijection with the set of pairs $(a,b)$ together with the set of triples $(a,b,E)$, where $a$ and $b$ specify a rational slope in standard form and $E$ specifies a tagged vertex set. 
Thus, listing the shear coordinates with respect to $T_0$ of these allowable curves without repetition amounts to listing one shear coordinate for each pair and triple. 
Proposition~\ref{prop: positive shear} handles the curves with non-negative or infinite slopes that are either closed or have a spiral point at $v_{00}$. 
We deal now with the remaining curves, beginning by removing the condition of a spiral point at $v_{00}$.

A lift of every non-closed curve of non-negative or infinite slope without a spiral point at $v_{00}$ may be obtained from a lift of one of the non-closed curves of Proposition~\ref{prop: positive shear} by a translation $[0,1], [1,0]$, or $[1,1]$. 
These translations preserve the triangulation $\overline{T}_0$ of $\reals^2$, simply permuting the labels of arcs. 
As a result, the maps affect shear coordinate vectors through a permutation of coordinates: $(14)(36), (25)(36)$, and $(14)(36)\cdot(25)(36)=(14)(25)$, respectively.
We claim that by applying the cyclic group $X = \langle (14)(25)(36) \rangle$ to the vectors listed in Item 1 and Item 4 of Proposition~\ref{prop: positive shear}, and applying the group $Y=\langle (14)(36),(25)(36) \rangle$ to the vectors listed in Item 2, with $b/a$ varying over all standard forms of rational slopes in the specified ranges for each, we obtain a complete and irredundant list of shear coordinates for allowable non-closed curves of non-negative or infinite slope. 
This follows from the following observations:
\begin{enumerate}
\item For any standard form $b/a$, there are two choices for the set $\{v\notch_{pq}, v\notch_{p+a, q+b}\}$. Since translation preserves spiral direction, for each curve handled in Item 3 of Proposition~\ref{prop: positive shear}, application of our three translations produces the lift of only one additional distinct curve: the other curve with slope $b/a$ and counterclockwise spirals at each spiral point. Analogously, application of these three translations to each curve handled in Item 4 of Proposition~\ref{prop: positive shear} also only produces the lift of one new curve. 
\item On the other hand, each of the 3 translations produces a distinct curve when applied to the differently-spiraled curves 
(with counterclockwise spiral at $v_{00}$)
handled in Item 2 of Proposition~\ref{prop: positive shear}, and a lift of every differently-spiraled curve of slope $b/a$ may be obtained in this manner.
\item In particular, translation of the plane by $[a,b]$ with respect to a lift of $\curve(a,b, \{v\notch_{00}, v_{ab}\})$ yields a lift of $\curve(a,b,\{v_{00}, v\notch_{ab}\})$. This accounts for our omitting Item 3 in Proposition~\ref{prop: positive shear} from Theorem~\ref{thm: shear}.
\item In the formulas of Proposition~\ref{prop: positive shear}, pairs of coordinates corresponding to parallel arcs in $T_0$ (namely, 14, 25, and 36) are either equal or differ by 1, depending upon $a$ and $b$ modulo two. 
\item In the vectors listed in Item 1 and Item 4 of Proposition~\ref{prop: positive shear}, exactly one pair of coordinates differs for each choice of standard form slope $b/a$. That is, exactly one of the transpositions in the coordinate permutation $(14)(25)(36)$ acts non-trivially, and the result is the shear coordinates for the other counterclockwise- (respectively, clockwise-) spiraled curve of slope $b/a$. 
\item In the vectors listed in Item 2 of Proposition~\ref{prop: positive shear}, exactly two pairs of coordinates differ for each choice of standard form slope $b/a$. Swapping one, the other, or both of these pairs of coordinates results in the shear coordinates for the other three differently-spiraled curves of slope $b/a$. Case-by-case computation confirms that the application of the group $Y$ corresponds to these permutations. 
\end{enumerate}
 
We proceed now to the allowable curves of negative slope.
The linear map $[a,b] \mapsto [a+b, -a]$ restricts to a bijection from $\{[a,b] : a,b\ge0\}$ to ${\set{[a,b]:-a\le b\le 0}}$, mapping curves with nonnegative or infinite slope $b/a$ to curves with negative slope $-1 \le -a/(a+b) \le 0$. 
Similarly, the map $[a,b] \mapsto [b,-a-b]$ restricts to a bijection from $\{[a,b] : a,b\ge0\}$ to $\{[a,b] : 0\le a\le -b\}$, mapping curves with nonnegative or infinite slope $b/a$ to curves with negative slope $-\infty \le (-a-b)/b \le -1$.   
(The notation $-\infty$ serves to make the inequality meaningful, but of course we do not distinguish between positive and negative infinite slopes.) 
Both maps preserve spiraling direction since they rotate and stretch the plane but do not reflect, and also preserve the triangulation $\overline{T}_0$ of $\reals^2$, simply permuting the labels of arcs. 
As a result, these linear maps affect shear coordinate vectors through a permutation of coordinates. 
For example, if $-a\le b\le0$ and the shear coordinates of $\curve(-b,a+b, \{v\notch_{00}, v\notch_{b,a+b}\})$ are $[x_1,x_2,x_3,x_4,x_5,x_6]$, then the shear coordinates of $\curve(a,b, \{v\notch_{00}, v\notch_{ab}\})$ are $[x_2,x_3,x_1,x_5,x_6,x_4]$.
(See Example~\ref{ex: shear5} for a concrete illustration.)
This happens because the map $[a,b] \mapsto [a+b, -a]$ sends the lifts of horizontal arcs to lifts of diagonal arcs, lifts of vertical arcs to lifts of horizontal arcs, and lifts of diagonal arcs to lifts of vertical arcs, corresponding to the coordinate permutation $[(123)(456)]^2$.
Similarly, the map $[a,b] \mapsto [b, -a-b]$ corresponds to the coordinate permutation $(123)(456)$.

Thus by letting $b/a$ vary over all standard forms of rational slopes in $(0, \infty]$, constructing the vectors in Proposition~\ref{prop: positive shear}, and then applying the cyclic group $Z = \langle (123)(456) \rangle$, we obtain a complete list of shear coordinates for allowable curves that are either closed or have spiral points at $v_{00}$. The half-open range for $b/a$ guarantees that our list is irredundant. 
By first applying the cyclic group $X$ to the vectors in the first and fourth items and the group $Y$ to the vectors in the second item, and then applying $Z$ to the results (or equivalently, applying the permutations in the sets $Z \cdot X$ and $Z \cdot Y$) along with applying $Z$ to vectors in the fifth item, we obtain a complete and irredundant list of {\em all} shear coordinates.
\end{proof}

\begin{example}
\label{ex: shear5}
Consider $\lambda'' = \curve(5,-2,\{v\notch_{00},v\notch_{10}\})$ whose lift $\bar{\lambda}''$ is shown in the bottom picture of Figure~\ref{shear ex fig 2}, with shear coordinate vector $\b(T_0, \lambda'')=[2,0,-1,1,0,-1]$. 
Note that $\bar{\lambda}''$ is obtained from the lift of the curve $\lambda$ of Examples~\ref{ex: shear1} and ~\ref{ex: shear2} via application of the map $[-b, a+b] \mapsto [a,b]$, and $\b(T_0, \lambda'')$ is obtained from $\b(T_0, \lambda)$ by applying the coordinate permutation $[(123)(456)]^2=(132)(654)$.
\end{example}

\begin{example}
\label{ex: shear6}
The curve $\lambda'''= \curve(2,3,\{v\notch_{10},v\notch_{11}\})$ whose lift $\bar{\lambda'''}$ is shown in the top right picture of Figure~\ref{shear ex fig 2} has shear coordinate vector $\b(T_0, \lambda''')=[-1,1,0,-1,2,0]$. 
Note that $\bar{\lambda''}$ is obtained from the lift of $\lambda$ by a one unit translation to the right, and $\b(T_0, \lambda''')$ is obtained from $\b(T_0, \lambda)$ by applying the coordinate permutation $(14)(25)(36),$ where only the second transposition, $(25)$, acts non-trivially. 
\end{example}

\section{The Null Tangle Property}\label{NTP sec}  
With the classification of allowable curves and explicit shear coordinates complete, we prove Theorem~\ref{sphere NTP}. 
A  key tool in the proof is the Null Tangle Property for the once-punctured torus, which was shown to hold in \cite[Theorem~3.2]{unitorus}.
We begin by showing that null tangles in $\Sp-4$ are supported on closed curves.
\begin{proposition}\label{support_closed}
Fix a null tangle $\Xi$ in the four-punctured sphere.
The support of $\Xi$ contains only closed curves.
\end{proposition}

To assist in the proof of Proposition~\ref{support_closed}, we quote two results.
The following proposition is \cite[Proposition 7.11]{unisurface}, specialized to the case of the four-punctured sphere.
\begin{proposition}\label{separation}
Let $\Xi$ be a null tangle in the four-punctured sphere.
Suppose for some tagged triangulation~$T$, for some tagged arc $\gamma$ in $T$, and for some curve $\lambda$ in $\Xi$ that $b_{\gamma}(T,\lambda)$ is strictly positive and that $b_{\gamma}(T,\nu)$ is nonpositive for every other curve $\nu\in \Xi$. Then~$w_{\lambda} =0$.
\end{proposition}

The following lemma is an easy consequence of \cite[Lemma~6.2]{unitorus} (see the second paragraph of the proof of \cite[Proposition~6.1]{unitorus} for a detailed explanation).
\begin{lemma}\label{Farey_triple}
Given a finite set $M$ of slopes and a rational slope $f/e\in M$ in standard form, there exist rational slopes $b/a$ and $d/c$ (not necessarily in $M$) in standard form with the following properties:
First, $(b/a, d/c, f/e)$ is a Farey-1 triple, second, $b/a<d/c<f/e$, and third, there is no $q\in M$ with $b/a\le q<f/e$. 
\end{lemma}

To prove Proposition~\ref{support_closed}, we build a special type-I triangulation $\cT$ of $\Sp-4$ such that exactly one curve in $\Xi$ has strictly positive shear coordinates with respect to a particular arc in $\cT$, and appeal to Proposition~\ref{separation}.

\begin{proof}[Proof of Proposition~\ref{support_closed}]
Suppose that $\lambda$ is a non-closed curve in $\Xi$.
By Proposition~\ref{prop: allowable curves}, $\lambda$ is $\curve(e,f,E)$ for some rational slope with standard form $f/e$ and spirals clockwise or counterclockwise into its spiral points according to $E$.
We will show that the weight of $\lambda$ is zero. 
Let $M$ be the set of all rational slopes $s/r$ such that there is a curve in $\Xi$ with slope $s/r$.
Choose $b/a$ and $d/c$ so that $(b/a, d/c, f/e)$ has the properties listed in Lemma~\ref{Farey_triple}.
By Proposition~\ref{all tri}, the Farey-1 triple $(b/a,d/c,f/e)$ and a choice of tagging uniquely specifies a type-I triangulation $\mathcal{T}$ of $\Sp-4$.

In the triangulation $\cT$, let $\gamma_1$ and $\gamma_4$ denote the arcs with slope $d/c$, let $\gamma_2$ and $\gamma_5$ denote the arcs with slope $f/e$,  (with the endpoints of $\gamma_2$ coinciding with the spiral points of $\lambda$), and let $\gamma_3$ and $\gamma_6$ denote the arcs with slope $b/a$. 
For each endpoint $v$ where $\lambda$ spirals clockwise, we tag each of the arcs incident to $v$ notched at that endpoint.
We leave all other tags plain.
Since a notched tagging at the arcs incident to an endpoint of $\lambda$ changes the computation of the shear coordinate vector $\b(\cT,\lambda)$ by changing the spiral direction of $\lambda$, we proceed as though $\lambda$ spirals counterclockwise into both of its endpoints.
Figure~\ref{NT sphere fig} shows the triangulation $\cT$ and the arc $\lambda$ with counterclockwise spirals at both spiral points. 
\begin{figure}
\scalebox{1.3}{\includegraphics{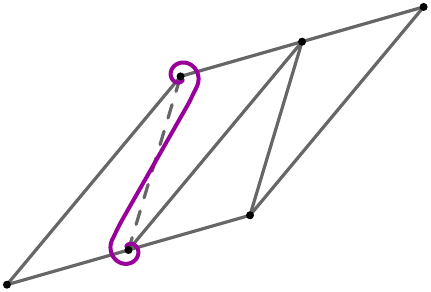}
\begin{picture}(0,0)(87, -12)
\definecolor{mypurple}{RGB}{153,0,153}
\put(17,42){\footnotesize\textcolor{mypurple}{$\lambda$}}
\put(-27,16){\footnotesize$\gamma_1$}
\put(1,18){\footnotesize$\gamma_2$}
\put(-21,-10){\footnotesize$\gamma_3$}
\end{picture}}
\caption{The triangulation $\cT$ from the proof of Proposition~\ref{support_closed}}
\label{NT sphere fig}
\end{figure}
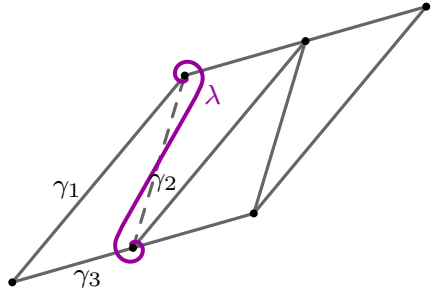
By inspection of Figure~\ref{NT sphere fig}, we conclude that $b_{\gamma_2}(\mathcal{T},\lambda)=1$.
The Farey-1 condition implies that for any pair of slopes in $(b/a,d/c,f/e)$, the corresponding pair of integer lattice points forms a basis for the integer lattice.
In particular, the lift of~$\cT$ to $\R^2-\Z^2$, depicted in Figure~\ref{NT lift tri}, is merely a relabeling of the integer lattice (as depicted in Figure~\ref{lift tri}) according to this change of basis.
Thus, we may assume without loss of generality that $\cT=T_0$ (where $\gamma_2$ corresponds to the arc labeled 2 in $T_0$).
Next, we argue that every other curve in $\Xi$ has non-positive shear coordinate with respect to $\gamma_2$.

\begin{figure}
\includegraphics[scale=0.9]{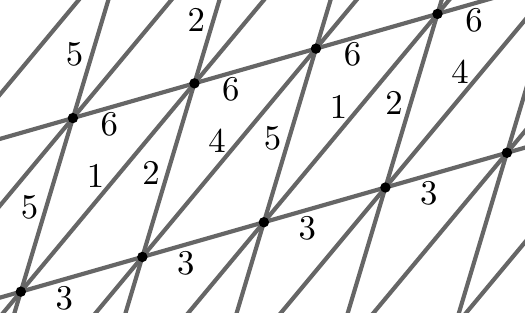}
\caption{The triangulation obtained by lifting $\cT$ (or $\cT'$) to the plane.}
\label{NT lift tri}
\end{figure}

Let $\lambda'$ be a curve in $\Xi$, distinct from $\lambda$, with rational slope $s/r$.
We claim that $b_{\gamma_2}(\cT,\lambda')$ is positive only if $b/a\le s/r\le f/e$.
Choose an isotopy representative of $\lambda'$ that minimizes the number of its intersections with the arc $\gamma_2$.
Recall from the proof of Theorem~\ref{thm: shear} that the shear coordinate vector of a curve with negative slope is obtained by applying either the  coordinate permutation $(123)(456)$ or $[(123)(456)]^2$ to the shear coordinate vector of a curve with nonnegative slope.
By Proposition~\ref{prop: positive shear}, we see that only the second permutation could yield a shear coordinate vector with $b_{\gamma_2}(\cT,\lambda')$ positive.
In that case, $[r,s]$ is in the image of the map which sends $[a,b]$ to $[a+b, -a]$, and thus $-1\le s/r\le 0$.
We conclude that $b_{\gamma_2}(\cT,\lambda')$ is positive only if $b/a\le s/r\le f/e$.

Because $(b/a, d/c, f/e)$ has the properties listed in Lemma~\ref{Farey_triple}, we may assume that $s/r=f/e$.
In this case, since $\lambda'$ does not equal $\lambda$, it is either closed or spirals clockwise into at least one of its spiral points.
If $\lambda'$ is closed or spirals clockwise into exactly one of its endpoints, then $b_{\gamma_2}(\mathcal{T},\lambda')=0$.
If $\lambda'$ spirals clockwise into both endpoints, then $b_{\gamma_3}(\mathcal{T},\lambda')=-1$. 
Thus $\lambda'$ has nonpositive shear coordinate with respect to $\gamma_2$.
By Proposition~\ref{separation}, we conclude that $w_{\lambda}=0$.
\end{proof}

To complete the proof of Theorem~\ref{sphere NTP}, we relate the shear coordinates of closed curves in $\Sp-4$ to the shear coordinates of closed curves in $\T-1$.
Suppose that $T'$ is a tagged triangulation of $\T-1$.
Since there is only one puncture, every taggable arc is a loop.
Thus an arc that is tagged notched at one endpoint must also be tagged notched at its other endpoint.
Moreover, if any single arc in $T'$ is tagged notched, then every other arc must also be  tagged notched at both endpoints.  
We will say that $T'$ is notched if all of its arcs are notched; similarly we say that $T'$ is plain if all of its arcs are marked plain.
The following proposition is \cite[Proposition~4.3]{unitorus}, and characterizes tagged triangulations in $\T-1$.
\begin{proposition}\label{torus:tri}
Each tagged triangulation $T'$ in the once-punctured torus is uniquely determined by a Farey-1 triple and a choice of plain or notched.
\end{proposition}

By Proposition~\ref{all tri}, we immediately obtain the following corollary.
\begin{corollary}\label{bij:torus_tri}
The map sending a tagged triangulations in $\T-1$ to the tagged triangulation in $\Sp-4$ with the same slopes and tagging is a bijection from the set of tagged triangulations in $\T-1$ to the set of type-\textrm{I} triangulations of $\Sp-4$ with either all tags plain or all tags notched.
\end{corollary}

Let $p:\R^2-\Z^2\to \T-1$ denote the covering of $\T-1$ by the integer punctured plane.
Let $\overline{\curve}(a,b)$ denote the projection of a straight line in $\R^2-\Z^2$ with slope $b/a$ in standard form, containing no integer points. 
Let $T_0'$ denote the plain tagged triangulation of $\T-1$ associated with the Farey-1 triple $(-1,0,\infty)$, depicted in Figure~\ref{tri and mat torus}. 
The following proposition is a combination of \cite[Proposition~4.4]{unitorus}, and \cite[Proposition~5.1]{unitorus}.
The computation of shear coordinates for closed curves in $\T-1$ can be found in the proof of \cite[Proposition~5.1]{unitorus}.

\begin{proposition}\label{torus:shear}
\quad  
\begin{enumerate}
\item The curve $\overline{\curve}(a,b)$ is an allowable closed curve in the once punctured torus. 
Each allowable closed curve in the once punctured torus is of this form.
\item The shear coordinates  of $\overline{\curve}(a,b)$ with respect to $T_0'$ are obtained by a cyclic permutation of the vector $[-b,a,b-a]$.
\end{enumerate}
\end{proposition}
Fix a closed curve $\lambda = \curve(a,b)$ in $\Sp-4$ and closed curve $\lambda'=\overline{\curve}(a,b)$ in $\T-1$, and let $[x_1,x_2,x_3,x_4,x_5,x_6]$ denote the shear coordinate vector $\b(\lambda,T_0)$. 
Comparing the formula for $\b(\lambda,T_0)$ from Item 4 in Theorem~\ref{thm: shear} with the second item in Proposition~\ref{torus:shear}, we observe that $\b(\lambda',T_0') = [x_1,x_2,x_3]$.
We now extend this observation to all type-I triangulations.  
Fix a Farey-1 triple $(q_1,q_2,q_3)$, and let $\cT$ denote the tagged triangulation of $\Sp-4$ specified by $(q_1,q_2,q_3)$, with all arcs marked plain at every endpoint.
Denote the arcs in $\cT$ by $\gamma_1, \gamma_2, \ldots \gamma_6$, such that $\gamma_i$ and $\gamma_{i+3}$ are a Farey-0 pair with slope $q_i$, for $i=1,2,3$.
Let $\cT'$ denote the plain tagged triangulation of $\T-1$ associated with $(q_1,q_2,q_3)$.
Let $\{\gamma_i': i=1,2,3\}$ be the arcs in $\cT'$, with $\gamma_i'$ having slope $q_i$.
Consider the set $\mathcal{S}$ of all of the lifts of the arcs in $\cT$ to $\R^2-\Z^2$.
The set $\mathcal{S}$ consists of all integer translations of line segments with slopes $q_1$, $q_2$, and $q_3$, with integer endpoints.
This is exactly the set of all of the lifts (under $p$) of the arcs in $\cT'$ to $\R^2-\Z^2$.
As in the proof of Proposition~\ref{support_closed}, the triangulation of the plane that we obtain from $\mathcal{S}$ is a relabeling of the triangulation depicted in Figure~\ref{lift tri}, given by a suitable change of basis (see Figure~\ref{NT lift tri}).
Thus we can compute shear coordinates with respect to $\cT$ and $\cT'$ by the method used in the previous section.
The following proposition is now immediate.  

\begin{proposition}\label{sphere_to_torus}
Suppose that $\cT'=\{\gamma_1',\gamma_2',\gamma_3'\}$ and $\cT=\{\gamma_1,\ldots \gamma_6\}$ are tagged triangulations of $\T-1$ and $\Sp-4$ respectively, associated to the Farey-1 triple $(q_1,q_2,q_3)$, with arcs indexed as above.
Fix closed curves $\lambda'=\overline{\curve}(a,b)$ and $\lambda=\curve(a,b)$ in $\T-1$ and $\Sp-4$, respectively.
Then $b_{\gamma_i}(\mathcal{T},\lambda) = b_{\gamma_i+3}(\mathcal{T},\lambda) = b_{{\gamma_i'}}(\cT',\lambda')$ for $i \in \{1,2,3\}$.
\end{proposition}
We apply Proposition~\ref{sphere_to_torus} and the Null Tangle Property for the once-punctured torus to complete the proof of Theorem~\ref{sphere NTP}.
\begin{proof}[Proof of Theorem~\ref{sphere NTP}] 
Let $\Xi$ be a null tangle in $\Sp-4$.
By Proposition~\ref{support_closed}, we may assume that the support of $\Xi$ contains only closed curves.
Define a tangle $\Xi'$ in $\T-1$ consisting of the curves $\overline{\curve}(a,b)$ such that $\curve(a,b)$ is in $\Xi$, and such that the weight of $\overline{\curve}(a,b)$ in $\Xi'$ is the same as the weight of $\curve(a,b)$ in $\Xi$.
Suppose that $\cT'$ is any triangulation of $\T-1$, associated with the Farey-1 triple $(q_1,q_2,q_3)$. 
Let $\cT$ be a type-\textrm{I} triangulation of $\Sp-4$ associated with the same Farey-1 triple, tagged arbitrarily, with each arc indexed as in Proposition~\ref{sphere_to_torus}.
Fix a closed curve $\lambda'=\overline{\curve}(a,b)$ in $\Xi'$ and its corresponding curve $\lambda = \curve(a,b)$ in $\Xi$.
By Proposition~\ref{sphere_to_torus}, the shear coordinate vector $\b(\cT',\lambda')$ is  obtained from $\b(\cT,\lambda)$ by the projection that maps the vector $[b_1,b_2,b_3,b_4,b_5,b_6]\in\reals^6$ to $[b_1,b_2,b_3]\in\reals^3$.
Thus, $\Xi'$ is a null tangle because $\Xi$ is a null tangle.
Since $\T-1$ is known to have the Null Tangle Property \cite[Theorem 3.2]{unitorus}, $w_{\lambda'}$ is zero.
Since $w_\lambda$ is equal to $w_{\lambda'}$, we conclude that $\Xi$ is trivial.
\end{proof}

With the Null Tangle Property established, Theorem~\ref{Null Tangle consequences} says that the set of all shear coordinate vectors with respect to $T_0$ constitutes universal geometric coefficients for $B(T_0)$ over $\Z$ and $\Q$. 
Thus Theorem~\ref{unisphere thm} follows from Theorem~\ref{thm: shear} and Proposition~\ref{same list}.

\section{The mutation fan}\label{global sec} 
Proposition~\ref{prop: maximal curves} and the ensuing discussion in Section~\ref{curve sec} provide significant information about the rational quasi-lamination fan $\F_\rationals(T_0)$ and thus, by Theorem~\ref{Null Tangle consequences}, the mutation fan $\F_{B(T_0)}$.
However, there is still an important question about each fan that we cannot answer.
First, we don't know the support $|\F_\rationals(T_0)|$ of the rational quasi-lamination fan (the union of all of the cones in $\F_\rationals(T_0)$).
Second, while we know the rational part of $\F_{B(T_0)}$, we do not know all of the cones in $\F_{B(T_0)}$.
In this section, we conjecture answers to these questions.
Furthermore, we make a conjecture on real universal geometric coefficients that would follow from the conjectures about the fans.

\begin{conjecture}\label{support}
For $T_0$ as shown in Figure~\ref{tri and mat}, the complement $\reals^6\setminus|\F_\rationals(T_0)|$ of the support of $\F_\rationals(T_0)$ is the union of the relative interiors of all irrational rays in the plane $P$ given by the equations $x_1+x_2+x_3=0$, $x_1=x_4$, $x_2=x_5$, and $x_3=x_6$.
\end{conjecture}

\begin{conjecture}\label{irrat rays}
The mutation fan $\F_B$, for $B=B(T_0)$ as shown in Figure~\ref{tri and mat}, is the union of $\F_\rationals(T_0)$ with the set of all irrational rays in the plane $P$ defined in Conjecture~\ref{support}.
\end{conjecture}

Conjectures~\ref{support} and~\ref{irrat rays} are easily seen to be equivalent.
Indeed, given Conjecture~\ref{support}, the unique complete fan containing $\F_\rationals(T_0)$ is the fan described in Conjecture~\ref{irrat rays}.
Conversely, if Conjecture~\ref{irrat rays} holds, then by Theorem~\ref{Null Tangle consequences}, the only irrational cones in $\F_B$ are the irrational rays in the plane $P$. 

Theorem~\ref{two conjs}, below, states that Conjecture~\ref{irrat rays} is the last piece needed to prove the following conjecture.

\begin{conjecture}\label{univ real}
For $B=B(T_0)$ as shown in Figure~\ref{tri and mat}, a complete list of universal geometric coefficients for $B$ \emph{over $\reals$} consists of the vectors described in Items $1$--$3$ of Theorem~\ref{unisphere thm} together with exactly one nonzero vector in $\rho$ for each ray $\rho$ contained in the plane $P$ given by $x_1+x_2+x_3=0$, $x_1=x_4$, $x_2=x_5$, and $x_3=x_6$.
\end{conjecture}Equivalently, universal geometric coefficients for $B$ over $\reals$ consist of all of the vectors described in Theorem~\ref{unisphere thm} together with exactly one nonzero vector in $\rho$ for each \emph{irrational} ray $\rho$ contained in $P$.

\begin{theorem}\label{two conjs}
If Conjecture~\ref{irrat rays} holds, then Conjecture~\ref{univ real} holds also.
\end{theorem}

\begin{proof}[Proof sketch]
Suppose Conjecture~\ref{irrat rays} holds.
The proof that Conjecture~\ref{univ real} holds is essentially a straightforward modification of the proof (in \cite[Section~8]{unitorus}) of the part of \cite[Theorem~1.1]{unitorus} that is analogous to Conjecture~\ref{univ real}.
However, we can shortcut the proof by actually \emph{using}, rather than modifying, the key subsidiary result in \cite{unitorus}.

We extend the collection of allowable curves to allow curves that are dense in $\Sp-4$, parametrized by irrational slopes.
Specifically, for each positive slope $\sigma$, choose a line in $\reals^2 - \integers^2$ having slope $\sigma$ and let $\lambda(\sigma)$ be the projection of that line to $\Sp-4$.
If $\sigma$ is rational, then $\lambda(\sigma)$ is a closed allowable curve in the usual sense.
In general, we can define \newword{normalized shear coordinates} for any $\lambda(\sigma)$ with $\sigma>0$
by starting at the projection of the point on the line with vertical coordinate zero and traversing the line a distance $d$ in the direction of increasing vertical coordinate.
Given a tagged triangulation of $\Sp-4$, we compute shear coordinates as we traverse $\lambda(\sigma)$.
The normalized shear coordinates are the limit, as $d\to\infty$, of the shear coordinates scaled to unit vectors.
Arguing as in \cite[Section~8]{unitorus}, we conclude that the normalized shear coordinates of $\lambda(\sigma)$ is the unit vector in the direction of $[-\sigma,1,\sigma-1,-\sigma,1,\sigma-1]$.
In particular, Proposition~\ref{sphere_to_torus} extends to curves with irrational slope $\sigma$.
We define a \newword{real tangle} in $\Sp-4$ to be a finite collection of allowable curves and/or irrational curves of the form $\lambda(\sigma)$ with arbitrary real weights.
A real tangle is \newword{null} if the weighted sum of its shear coordinates (using \emph{normalized} shear coordinates for irrational curves) is zero with respect to any tagged triangulation.
By the same argument as in \cite[Section~8]{unitorus}, it is enough to establish the Real Null Tangle Property:  that every real null tangle is trivial.

The Real Null Tangle Property can be established by first showing, using exactly the argument given here for ordinary (rational) tangles, that a real null tangle has no curves with spirals.
To show that a real null tangle without spirals is trivial, we can re-use the proof of Theorem~\ref{sphere NTP}, substituting the Real Null Tangle Property of $\T-1$, proved as \cite[Theorem~8.1]{unitorus}, for the Null Tangle Property of $\T-1$.
\end{proof}
Although the proof of Theorem~\ref{two conjs} involved establishing a ``Real Null Tangle Property'' for the four-punctured sphere, we emphasize, as was emphasized in \cite{unitorus}, that we have no definition of a real tangle in a general surface.

There is not much evidence for Conjecture~\ref{support}, beyond the fact that Theorem~\ref{two conjs} holds.
We conclude by offering one more piece of indirect evidence.

Given any linear subspace $U$ of $\reals^6$, the mutation fan $\F_B$ induces a fan structure $(\F_B)_U=\set{C\cap U:C\in \F_B}$ on $U$.

\begin{proposition}\label{induced}
Let $B=B(T_0)$ be as shown in Figure~\ref{tri and mat} and let $B'=B(T'_0)$ be as shown in Figure~\ref{tri and mat torus}.   
Let $U$ be the subspace of $\reals^6$ given by the equations $x_1=x_4$, $x_2=x_5$, and $x_3=x_6$.
The map $r$ that maps the vector $[b_1,b_2,b_3,b_1,b_2,b_3]\in\reals^6$ to $[b_1,b_2,b_3]\in\reals^3$ restricts to an isomorphism from the induced fan $(\F_B)_U$ to the mutation fan $\F_{B'}$.
\end{proposition}
\begin{proof}[Proof sketch]
Recall from Corollary~\ref{bij:torus_tri} the correspondence between triangulations $\mathcal T'$ of $\T-1$ and type-I triangulations $\mathcal T$ of $\Sp-4$ with all tags the same.
Recall also the map $\kappa$ defined in Section~\ref{defs sec}.
(We will use the same symbol $\kappa$ for this map in both surfaces.) 
We first check that, for each $\mathcal T$, the full-dimensional cones $C_{\kappa(\mathcal T)}$ of $\F_\rationals(T_0)$ and $C_{\kappa(\mathcal T')}$ of $\F_\rationals(T'_0)$ satisfy $C_{\kappa(\mathcal T')}\subseteq r(C_{\kappa(\mathcal T)}\cap U)$.
First suppose that $\mathcal T$ has plain tags, and consider two curves $\kappa(\gamma)$ and $\kappa(\gamma')$ in $\kappa(\mathcal T)$ with the same slope $b/a$.
If $b/a$ is nonnegative, then one of these curves has shear coordinates as described in Item 4 of Proposition~\ref{prop: positive shear}.
The shear coordinates of the other are obtained by applying the coordinate permutation $(14)(25)(36)$.
The sum of the two is contained in $U$ and maps, by $r$, to the shear coordinates for the curve of the same slope in $\kappa(\mathcal T')$, as recorded in \cite[Proposition~5.1]{unitorus}.
If $b/a$ is negative, then by applying $(123)(456)$ to shear coordinates of curves in $\Sp-4$ and $(123)$ to shear coordinates of curves in $\T-1$, we come to the same conclusion.
If $\mathcal T$ has notched tags, we argue similarly.
We see that each $C_{\kappa(\cT')}$
 has its extreme rays contained in $r(C_{\kappa(\mathcal T)}\cap U)$, so $C_{\kappa(\mathcal T')}\subseteq r(C_{\kappa(\mathcal T)}\cap U)$.

Recall from Proposition~\ref{prop: maximal curves} that a type-VII ($5$-dimensional) maximal cone in $\F_B$ is determined by a slope, by a choice of two marked points, and by a choice of tagging at the other two points.
A $2$-dimensional maximal cone of $\F_{B'}$ is determined by a slope and by a tagging at the marked point.
Similarly to the previous paragraph, we argue, using \cite[Proposition~5.1]{unitorus} and \cite[Theorem~7.1]{unitorus}, that each two-dimensional maximal cone $C'$ of $\F_{B'}$ is contained in $r(C\cap U)$, where $C$ is any $5$-dimensional (i.e.\ type-VII) maximal cone of $(\F_B)$ with the same slope as $C'$ and with both taggings in $C$ matching the tagging of $C'$.

Since $\F_\rationals(T'_0)$ covers all of $\reals^3$, except for the irrational rays in the plane given by $x_1+x_2+x_3=0$, and since all the cones in question are rational, we conclude that each containment $C_{\kappa(\mathcal T')}\subseteq r(C_{\kappa(\mathcal T)}\cap U)$ or $C'\subseteq r(C\cap U)$ is actually equality.
The result follows.
\end{proof}

\section*{Acknowledgments}
The authors thank
Dylan Allegretti,
Bill Floyd,
Vladimir Fock,
Alexander Goncharov,
Allen Hatcher, and
Lee Mosher
for helpful responses to email queries.

An extended abstract \cite{FPSAC} of this paper appeared in the proceedings of the 27th International Conference on Formal Power Series and Algebraic Combinatorics (\mbox{FPSAC} 2015) in Daejeon, South Korea, July 6--10, 2015.
\label{sec:ack}

\end{document}